\documentclass{article}
\usepackage{color}
\usepackage[normalem]{ulem}

\advance \textwidth -60truemm\relax

\advance\oddsidemargin -1truein\relax

\evensidemargin=\oddsidemargin

\DeclareFontFamily{U}{rsfs}{\skewchar\font"7F}
\DeclareFontShape{U}{rsfs}{m}{n}{
<-6> rsfs5
<6-8> rsfs7
<8-> rsfs10
}{}
\DeclareMathAlphabet{\mathscr}{U}{rsfs}{m}{n}

\usepackage{color}
\usepackage{amsmath,amssymb}

\DeclareMathAlphabet{\mathbbb}{U}{bbold}{m}{n}

\usepackage{bm}
\bmdefine{\aaa}{a}
\bmdefine{\bbb}{b}
\bmdefine{\ccc}{c}
\bmdefine{\ddd}{d}
\bmdefine{\fff}{f}
\bmdefine{\mmm}{m}
\bmdefine{\ppp}{p}
\bmdefine{\qqq}{q}
\bmdefine{\ttt}{t}
\bmdefine{\uuu}{u}
\bmdefine{\vvv}{v}
\bmdefine{\www}{w}
\bmdefine{\eee}{e}
\bmdefine{\sss}{s}
\bmdefine{\xxx}{x}
\bmdefine{\yyy}{y}
\bmdefine{\zzz}{z}
\bmdefine{\zerovec}{0}
\bmdefine{\onevec}{1}

\newcommand{\NNN}{\mathbb{N}}
\newcommand{\CCC}{\mathbb{C}}
\newcommand{\RRR}{\mathbb{R}}
\newcommand{\QQQ}{\mathbb{Q}}

\newcommand{\III}{\mathbb{I}}
\newcommand{\KKK}{\mathbb{K}}
\newcommand{\LLL}{\mathbb{L}}
\newcommand{\MMM}{\mathbb{M}}
\newcommand{\VVV}{\mathbb{V}}
\newcommand{\ZZZ}{\mathbb{Z}}

\newcommand{\AAAAA}{{\mathcal A}}

\newcommand{\CCCCC}{{\mathcal C}}

\newcommand{\MMMMM}{{\mathcal M}}
\newcommand{\OOOOO}{{\mathcal O}}

\newcommand{\XXXXX}{{\mathcal X}}

\newcommand{\msAAA}{\mathscr{A}}
\newcommand{\msCCC}{\mathscr{C}}
\newcommand{\msHHH}{\mathscr{H}}
\newcommand{\msIII}{\mathscr{I}}
\newcommand{\msJJJ}{\mathscr{J}}

\newcommand{\msOOO}{\mathscr{O}}
\newcommand{\msPPP}{\mathscr{P}}
\newcommand{\msQQQ}{\mathscr{Q}}

\newcommand{\msSSS}{\mathscr{S}}
\newcommand{\msTTT}{\mathscr{T}}
\newcommand{\msUUU}{\mathscr{U}}
\newcommand{\msVVV}{\mathscr{V}}

\newcommand{\GL}{{\mathrm{GL}}}

\newcommand{\glin}{{\mathrm{GL}}}

\newcommand{\diag}{{\mathrm{Diag}}}
\newcommand{\rank}{\qopname\relax o{rank}}
\newcommand{\grank}{{\mathrm{grank}}}

\newcommand{\trank}{{\mathrm{trank}}}
\newcommand{\trdeg}{\mathrm{tr.deg}}
\newcommand{\height}{\mathrm{ht}}
\newcommand{\grade}{\mathrm{grade}}
\newcommand{\depth}{\mathrm{depth}}
\newcommand{\supp}{\mathrm{supp}}
\newcommand{\cof}{\mathrm{Cof}}
\newcommand{\chara}{\mathrm{char}}
\newcommand{\define}{\mathrel{:=}}
\newcommand\fl{\mathrm{fl}}
\newcommand\lf{\mathrm{lf}}
\newcommand\lexeq{\mathrel{\leq_{\mathrm{lex}}}}
\newcommand\Aind{\msIII}

\newcommand{\cm}{Cohen-Macaulay}
\newcommand{\qdeg}{\mathrm{qdeg}\,}

\newcommand{\rone}{(${R}_1$)}
\newcommand{\stwo}{(${S}_2$)}
\newcommand{\ronestwo}{{\rm (${R}_1$)+(${S}_2$)}}
\newcommand{\ri}{(${R}_i$)}
\newcommand{\si}{(${S}_i$)}

\newcommand{\PP}{{P\!P}}

\newcommand{\afcr}{{absolutely full column rank}}

\newcommand{\Afcr}{{Absolutely full column rank}}

\newcommand{\transpose}{^\top}
\def\zerovec{{\mathbbb 0}}
\def\RRR{\mathbb{R}}
\def\RP{\mathbb{RP}}
\newcommand{\dom}{{\mathrm{dom}}}
\newcommand{\interior}{\qopname\relax o{int}}

\newtheorem{thm}{Theorem}

\newtheorem{example}[thm]{Example}
\newtheorem{lemma}[thm]{Lemma}
\newtheorem{cor}[thm]{Corollary}
\newtheorem{definition}[thm]{Definition}
\newtheorem{prop}[thm]{Proposition}
\newtheorem{remark}[thm]{Remark}

\numberwithin{thm}{section}

\newcounter{enumtemp}

\numberwithin{thm}{section}
\numberwithin{equation}{section}

\newcommand{\bigzerou}{\smash{\lower1.7ex\hbox{\bg 0}}}

\newcommand{\bigastu}{\smash{\lower1.7ex\hbox{\bg *}}}

\def\typicalrankR{\trank}

\def\SeidenbergCite#1{\textcolor{green}{\cite[point #1]{Seidenberg:1974}}}

\def\SeidenbergCite#1{{\cite[part #1]{Seidenberg:1974}}}

\def\delete{\textcolor{magenta}\bgroup\hbox{\rm ---delete---}}
\def\enddelete{\hbox{\rm ---end of delete---}\egroup}
\def\replace#1\by{\textcolor{green}{\hbox{\rm ---replace---}}\footnote{#1}}

\def\inserts{\textcolor{blue}\bgroup}
\def\endinserts{\egroup}

\newcommand{\comment}[1]{\textcolor{green}{[#1]}}

\def\mylabel#1{\label{#1}}

\renewcommand{\comment}[1]{}

\makeatletter
\newcommand{\mysloppy}{\tolerance 9999 \hfuzz .5\p@ \vfuzz .5\p@}
\makeatother
\title{%
Typical ranks for $3$-tensors, nonsingular bilinear maps and determinantal ideals%
}

\author{Toshio Sumi, Mitsuhiro Miyazaki and Toshio Sakata}
\date{Version of \today}
\date{}

\begin{document}

\mysloppy


\maketitle

\begin{abstract}
Let $m,n\geq 3$, $(m-1)(n-1)+2\leq p\leq mn$, and $u=mn-p$.
The set $\RRR^{u\times n\times m}$ of all real tensors with size 
$u\times n\times m$ is one to one corresponding to the set of bilinear maps 
$\RRR^m\times \RRR^n\to \RRR^u$.
We show that $\RRR^{m\times n\times p}$ has plural typical ranks $p$ and $p+1$ if and only if
there exists a nonsingular bilinear map $\RRR^m\times\RRR^n\to\RRR^{u}$.
We show that there is a dense open subset $\msOOO$ of $\RRR^{u\times n\times m}$ such that
for any $Y\in\msOOO$, the ideal of maximal minors of a matrix defined by $Y$ in a certain way
is a prime ideal and the real radical of that is the irrelevant maximal ideal
if that is not a real prime ideal.
Further, we show that there is a dense open subset
$\msTTT$ of $\RRR^{ n\times  p \times m}$ and 
continuous surjective open 
maps $\nu\colon\msOOO\to\RRR^{u\times p}$  
and $\sigma\colon\msTTT\to\RRR^{u\times p}$,
where $\RRR^{u \times p}$ is the set of $u\times p$ matrices with entries in $\RRR$,
 such that 
if $\nu(Y)=\sigma(T)$, then
$\rank T=p$ if and only if 
the ideal of maximal minors of the matrix defined by 
$Y$ is a real prime ideal.
\end{abstract}

\section{Introduction}

For positive integers $m$, $n$, and $p$, we consider an $m\times n\times p$ tensor which is an element of the tensor product of $\RRR^m$, $\RRR^n$, and $\RRR^p$ with standard basis.  This tensor can be identified with a $3$-way array $(a_{ijk})$ where
$1\leq i\leq m$, $1\leq j\leq n$ and $1\leq k\leq p$.
We denote by $\RRR^{m\times n\times p}$ the set of all $m\times n\times p$ tensors.  This set is a topological space with Euclidean topology.
Hitchcock \cite{Hitchcock:1927} defined the rank of a tensor.
An integer $r$ is called a typical rank of $\RRR^{m\times n\times p}$ if
the set of tensors with rank $r$ is a semi-algebraic set of dimension $mnp$.
In the other words, $r$ is a typical rank of $\RRR^{m\times n\times p}$ if
the set of tensors with rank $r$ contains an open set of $\RRR^{m\times n\times p}$.
In this paper we discuss the typical ranks of $3$-tensors and connect between plurality of typical ranks and existence of a nonsingular bilinear map.

Let $n\leq p$.
A typical rank of $\RRR^{1\times n\times p}$ is equal to an $n\times p$ 
matrix full rank, that is, $n$.
If $n\geq 2$, then 
the set of typical ranks of $\RRR^{2\times n\times p}$ is equal to 
$\{n,n+1\}$ if $n=p$ and otherwise ${\min\{p,2n\}}$ \cite{tenBerge-Kiers:1999}. 
This is also obtained from the equivalent class: almost all $2\times n\times p$
tensors are equivalent to $((E_n,O_{n\times (p-n)});(O_{n\times (p-n)},E_n))$
which has rank ${\min\{p,2n\}}$ if $n<p$ (see \cite{JaJa:1979b} or \cite{Sumi-etal:2009}),
see Section \ref{sec:nonsingular} for notation. 
Suppose that $n\geq m\geq 3$.
The set of typical ranks of $\RRR^{m\times n\times p}$ 
is equal to ${\min\{p,mn\}}$ if $(m-1)n<p$ \cite{tenBerge:2000}.
If $p=(m-1)n$ then the set of typical ranks of $\RRR^{m\times n\times p}$ 
depends on the existence of a nonsingular bilinear map $\RRR^m\times\RRR^n\to\RRR^n$: It is equal to $\{p\}$ if there is no nonsingular bilinear map $\RRR^m\times\RRR^n\to\RRR^n$ and $\{p,p+1\}$ otherwise \cite{Sumi-etal:2015}.
Here, a bilinear map $f\colon\RRR^m\times \RRR^n\to\RRR^r$ 
is called nonsingular if $f(\xxx,\yyy)=\zerovec$ implies
$\xxx=\zerovec$ or $\yyy=\zerovec$.

Suppose that $(m-1)(n-1)+1\leq p\leq (m-1)n$.
A typical rank of $\RRR^{m\times n\times p}$ is unknown except a few cases.
First, $p$ is a minimal typical rank, since $p$ is a generic rank of $\CCC^{m\times n\times p}$ 
{\cite{Catalisano-Geramita-Gimigliano:2002}}.  
The authors
\cite{Sumi-etal:2013} showed that the Hurwitz-Radon function gives a condition that $\RRR^{m\times n\times (m-1)n}$ has plural typical ranks.
We \cite{Miyazaki-etal:2012a} also showed that $\RRR^{m\times n\times p}$ has plural typical ranks for some $(m,n,p)$ by using the concept of absolutely full column rank tensors.
We let $m\# n$ be the minimal integer $r$ such that there is
a nonsingular bilinear map $\RRR^m\times \RRR^n\to\RRR^r$. 
Then $m\# n\leq m+n-1$ (see Section~\ref{sec:nonsingular}).
The set $\RRR^{r\times m\times n}$ of $r\times m\times n$ tensors
is one to one corresponding to the set of bilinear maps
$\RRR^m\times \RRR^n\to \RRR^r$. 
By this map the set of \afcr{} tensors is one to one corresponding to
the set of nonsingular bilinear maps.

\begin{thm}\mylabel{thm:main0}
Let $m,n\geq3$ and $(m-1)(n-1)+1\leq p\leq mn$.
\begin{enumerate}
\item
\mylabel{item:01}
If there exists a nonsingular bilinear map $\RRR^m\times\RRR^n\to\RRR^{mn-p}$,
then $\RRR^{m\times n\times p}$ has plural typical ranks.
\item
\mylabel{item:02}
If $p\geq (m-1)(n-1)+2$ and $\RRR^{m\times n\times p}$ has plural typical ranks,
then there exists a nonsingular bilinear map $\RRR^m\times\RRR^n\to\RRR^{mn-p}$.
\end{enumerate}
\end{thm}

\ref{item:01} of Theorem \ref{thm:main0}
is an extension of one of \cite{Miyazaki-etal:2012a}.
Furthermore, we completely determine the set $\trank(m,n,p)$ of typical ranks of $\RRR^{m\times n\times p}$ for $p{\geq}(m-1)(n-1)+2$ by the number $m\# n$.

\begin{thm} \mylabel{thm:main1}
Let $m,n\geq3$, $k\geq 2$, and $p=(m-1)(n-1)+k$. 
The set of typical ranks of $\RRR^{m\times n \times p}$ is given as follows.
$$\typicalrankR(m,n,p)=\begin{cases}
\{p,p+1\}, & 2\leq k\leq m+n-1-(m\# n) \\
\{p\}, & \max\{2,(m+n)-(m\# n)\}\leq k\leq m+n-2 \\
\{mn\}, & k\geq m+n-1. \\
\end{cases}$$
\end{thm}

Consider the case where $p=(m-1)(n-1)+1$, 
Friedland \cite{Friedland:2012} showed that $\RRR^{n\times n\times ((n-1)^2+1)}$ has plural typical ranks. We extend this result.

\begin{thm}\mylabel{thm:(m-1)(n-1)+1}
Let $m,n\geq3$ and $p=(m-1)(n-1)+1$.
$\RRR^{m\times n\times p}$ has plural typical ranks if 
$m-1$ and $n-1$ are not bit-disjoint.
\end{thm}

This article is organized as follows.
Sections~\ref{sec:nonsingular}--\ref{sec:cor} are preparation to show the above theorems.  
In Section~\ref{sec:nonsingular}, we set notations and discuss the number $m\# n$.
In Section~\ref{sec:afcr tensors}, we study \afcr{} tensors. 
Since the set of \afcr{} tensors is an open set, there exists a special form of an \afcr{} tensor if an \afcr{} tensor exists.
In Section~\ref{sec:ideals}, we state the other notions and deal with ideals of minors of matrices. 
Theorem~\ref{thm:it real} in Section~\ref{sec:ideals} which corresponds with the real radical ideals is quite interesting in its own right. 
We show that for integers with $0<t\leq{\min\{u,n\}}$ and $m\geq (u-t+1)(n-t+1)+2$,
there exist open subsets $\msOOO_1$ and $\msOOO_2$ of $\RRR^{u\times n\times m}$ such that the union of them is dense, $\III(\VVV(I_t(M(\xxx,Y))))=I_t(M(\xxx,Y))$ for $Y\in \msOOO_1$ and $\III(\VVV(I_t(M(\xxx,Y))))=(x_1, \ldots, x_m)$
for $Y\in \msOOO_2$,
where $I_t(M(\xxx,Y))$ is the ideal generated by all $t$-minors of the $u\times n$ matrix $M(\xxx,Y)=\sum_{k=1}^m x_kY_k$ given by the indeterminates 
$x_1,\ldots,x_m$ and $Y=(Y_1;\ldots;Y_m)\in\RRR^{u\times n\times m}$.
From this, we can give a subset of $m\times n\times p$ tensors
with rank $p$ for $3\leq m\leq n$ and $(m-1)(n-1)+2\leq p\leq (m-1)n$.
In Section~\ref{sec:mpo} we discuss a property for the determinantal ideals by using monomial 
{preorder}.  This property plays an important role for proving Theorem~\ref{thm:main0}.
We characterize $m\times n\times p$ tensors with rank $p$ in Section~\ref{sec:rank p}.
In Section~\ref{sec:cor}, we show that the existence of an \afcr{} tensor
with suitable size implies that $p+1$ is a typical rank of 
$\RRR^{m\times n\times p}$.
Moreover there exist a nonempty open subset $\msTTT_1$ consisting of tensors with rank $p$ and a possibly empty open subset $\msTTT_2$ consisting of tensors with rank greater than $p$, corresponding $\msOOO_1$ and $\msOOO_2$ respectively, such that the union of them is a dense subset of $\RRR^{m\times n\times p}$ (see Theorem~\ref{thm:rank}).
Finally, in Section~\ref{sec:upperbound}, we show that $p+2$ is not a typical rank of $\RRR^{m\times n\times p}$ and complete proofs of the above theorems.

\section{Nonsingular bilinear maps \mylabel{sec:nonsingular}}

We first recall some basic facts and establish terminology.

\begin{notation}
\begin{enumerate}
\item We denote by $\RRR^n$ (resp. $\RRR^{1\times n}$) the set of
$n$-dimensional column (resp. row) real vectors and by $E_n$ the $n\times n$ identity matrix. Let
$\eee_j$\index{$\eee_j$} be the $j$-th column vector of an identity matrix.

\item
For a tensor $x\in\RRR^n\otimes\RRR^p\otimes\RRR^m$ with
$x=\sum_{ijk}a_{ijk}\eee_i\otimes\eee_j\otimes\eee_k$,
we identify $x$ with
$T=(a_{ijk})_{1\leq i\leq n,1\leq j\leq p,1\leq k\leq m}$
and denote it by $(A_1;\ldots;A_m)$\index{$(A_1;\ldots;A_m)$},
where $A_k=(a_{ijk})_{1\leq i\leq n,1\leq j\leq p}$ for $k=1, \ldots, m$ 
is an $n\times p$ matrix,
and call $(A_1;\ldots;A_m)$ a tensor.

\item
We denote the set of $n\times p\times m$ tensors by
$\RRR^{n\times p\times m}$ and the set of typical ranks by $\typicalrankR(n,p,m)$\index{$\typicalrankR(n,p,m)$}.

\item
For an $n\times p\times m$ tensor $T=(T_1;\ldots;T_m)$,
an $l\times n$ matrix $P$
and an $k\times p$ matrix $Q$,
we denote by $PT$\index{$PT$} the $l\times p\times m$ tensor
$(PT_1;\ldots;PT_p)$ and by $TQ\transpose$\index{$TQ\transpose$} the $n\times k\times m$ tensor
$(T_1Q\transpose;\ldots;T_pQ\transpose)$. 

\item
For $n\times p$ matrices $A_1, \ldots, A_m$,
we denote by $(A_1, \ldots, A_m)$\index{$(A_1, \ldots, A_m)$} the $n\times mp$ matrix obtained by aligning
$A_1, \ldots, A_m$ horizontally.

\item
We set $\diag(A_1,A_2,\ldots,A_t)\index{$\diag(A_1,A_2,\ldots,A_t)$}=
\begin{pmatrix} A_1&&&O \\ 
  & A_2\cr &&\ddots \\
  O &&& A_t \end{pmatrix}
$ 
for matrices $A_1, A_2, \ldots, A_t$.

\item
For an $m\times n$ matrix $M$,
we denote by $M_{\leq j}$\index{$M_{\leq j}$} (resp.\ ${}_{j<}M$\index{${}_{j<}M$}) the $m\times j$ matrix consisting of the first $j$ (resp.\ last $n-j$) columns of $M$.
We denote by $M^{\leq i}$\index{$M^{\leq i}$} (resp.\ ${}^{i<}M$\index{${}^{i<}M$}) the $i\times n$ 
(resp.\ $(m-i)\times n$) matrix consisting of the first $i$ 
(resp.\ last $m-i$) rows of $M$.
We put $M^{<i}=M^{\leq i-1}$\index{$M^{<i}$} $M_{<i}=M_{\leq i-1}$\index{$M_{<i}$}, and $M^{=i}={}^{i-1<}(M^{\leq i})$\index{$M^{=i}$} 
which is the $i$-th row vector of $M$.

\item \mylabel{def:h fl sigma}
We set
$\fl_1(T)=(T_1,\ldots,T_m)$ and $\fl_2(T)=\begin{pmatrix} T_1\\ \vdots\\ T_m\end{pmatrix}$ for a tensor $T=(T_1;\ldots;T_m)$.

%
\end{enumerate}
\end{notation}

\begin{definition}\rm
\mylabel{def:nonsing bilin}
A bilinear map $f\colon \RRR^m\times \RRR^n\to\RRR^l$ is called
nonsingular if $f(\xxx,\yyy)=\zerovec$ implies $\xxx=\zerovec$ or $\yyy=\zerovec$.
For positive integers $m$ and $n$, we set
$$
m\#n\define\min\{l\mid \text{there exists a nonsingular bilinear map 
$\RRR^m\times \RRR^n\to\RRR^l$}\}.
$$\index{$m\#n$}
\end{definition}

Let $g\colon \RRR^{1\times u}\times \RRR^{1\times v}\to\RRR^{1\times (u\#v)}$ be a nonsingular bilinear map.
For positive integers $m$ and $n$,  let $f\colon \RRR^{1\times mu}\times\RRR^{1\times nv}\to\RRR^{1\times (m+n-1)(u\#v)}$
be a map defined by
$f((\aaa_1,\ldots,\aaa_m),(\bbb_1,\ldots, \bbb_n))
=(g(\aaa_1,\bbb_1),g(\aaa_1,\bbb_2)+g(\aaa_2,\bbb_1), \ldots,
\sum_{i+j=k}g(\aaa_i,\bbb_j), \ldots, g(\aaa_m,\bbb_n))$.
It is easily verified that $f$ is a nonsingular bilinear map.
Thus we have the following:

\begin{lemma}
\mylabel{lem:nonsing comp}
$(mu)\#(nv)\leq (m+n-1)(u\#v)$.
\end{lemma}

By applying this lemma to nonsingular bilinear maps
obtained by
multiplications of $\RRR$, $\CCC$, quaternions and octanions respectively,
we have the following:

\begin{prop}[{cf. \cite[Proposition 12.12 (3)]{Shapiro:2000}}] 
\mylabel{prop:4m-4n}
For $k=1$, $2$, $4$ and $8$,
it holds that $km\# kn \leq k(m+n-1)$.
\end{prop}

Let $\msHHH(r, s, n)$\index{$\msHHH(r, s, n)$} be 
the condition on the binomial coefficients, called the Stiefel-Hopf criterion, 
that the binomial coefficient $\binom{n}{k}$ is even whenever $n-s<k<r$.
If there exists a continuous, nonsingular, biskew map $\RRR^r\times\RRR^s\to\RRR^n$ then the Stiefel-Hopf criterion $\msHHH(r, s, n)$ holds.
Put 
$$
r \circ s = \min\{n \mid \msHHH(r, s, n) \text{ holds}\}.
\index{$r \circ s = \min\{n \mid \msHHH(r, s, n) \text{ holds}\}$}
$$
We have
$$\max\{r, s\} \leq r \circ s \leq r\# s \leq r + s - 1.$$
Putting $n^\ast=\lceil \frac{n}{2}\rceil$ for $n \in \ZZZ$,
the number $r\circ s$ is easily obtained by the formula
$$r\circ s =\begin{cases} 2(r^\ast \circ s^\ast)-1 & 
\text{if $r$, $s$ are both odd and $r^\ast\circ s^\ast=r^\ast+s^\ast-1$},\\
2(r^\ast\circ s^\ast) & \text{otherwise}
\end{cases}$$
(cf. \cite[Proposition~12.9]{Shapiro:2000}).

For a positive integer $n$, we put integers $\alpha_j(n)=0,1$, $j\geq 0$
such that $n=\sum_{j=0}^\infty \alpha_j(n)2^j$ is the dyadic expansion of
$n$ and let $\alpha(n)\define\sum_{j=0}^\infty \alpha_j(n)$\index{$\alpha(n)\define\sum_{j=0}^\infty \alpha_j(n)$} be the number of ones in the dyadic expansion of $n$.
Two integers $m$ and $n$ are bit-disjoint if
$\{j\mid \alpha_j(m)=1\}$ and $\{j\mid \alpha_j(n)=1\}$ are disjoint. 
For $k>h$, let $\tau(k,h)$\index{$\tau(k,h)=\#\{j\geq 0\mid \alpha_j(k-h)=0, \alpha_j(k)\neq\alpha_j(h)\}$} be a nonnegative number defined as
$$\tau(k,h)=\#\{j\geq 0\mid \alpha_j(k-h)=0, \alpha_j(k)\neq\alpha_j(h)\}.$$

\begin{prop} \mylabel{prop:bit-disjoint}
$r\# s =r+s-1$ if and only if $r-1$ and $s-1$ are bit-disjoint.
\end{prop}

\begin{proof}
If $r-1$ and $s-1$ are bit-disjoint, then 
$r\circ s=r\# s=r+s-1$ (cf. \cite[p. 257]{Shapiro:2000}).
Moreover, $\tau(k,h)=0$ if and only if $h$ and $k-h$ are bit-disjoint.
There is a nonsingular bilinear map  $\RRR^{h+1}\times\RRR^{k-h+\tau(k,h)}\to\RRR^{k}$ for $k>h\geq 0$ \cite{Lam:1968a} and thus
$(h+1)\# (k-h+\tau(k,h))\leq k$.
Putting $r=h+1$ and $k=r+s-2$, we have $r\# (s-1+\tau(r+s-2,r-1))\leq r+s-2$.
In particular, if $r-1$ and $s-1$ are not bit-disjoint then
$r\# s\leq r+s-2$.
\end{proof}

Let $\rho$ be the Hurwitz-Radon function defined as
$\rho(n)=2^b+8c$ for nonnegative integers $a,b,c$ 
such that $n=(2a+1)2^{b+4c}$ and $0\leq b<4$.
There is a nonsingular bilinear map $\RRR^n\times \RRR^{\rho(n)}\to \RRR^n$ \cite{Hurwitz:1922,Radon:1922}
and is no nonsingular bilinear map $\RRR^n\times \RRR^{\rho(n)+1}\to \RRR^n$
for any $n\geq 1$ \cite{Adams:1962}.
Therefore, $n\#\rho(n)\leq n$ and $n\#(\rho(n)+1)>n$.

\begin{cor} \mylabel{cor:n-by-n}
$n\# n\leq 2n-2$.
In particular, the equality $n\# n = 2n-2$ holds for $n=2^a+1$.
\end{cor}

\begin{proof}
The inequality $n\# n\leq 2n-2$ is clear by Proposition~\ref{prop:bit-disjoint} since $n-1$ and $n-1$ are not bit-disjoint.

There is an immersion
$\RP^{n}\to\RRR^{n+k}$ if and only if
there is a nonsingular biskew map $\RRR^{n+1}\times\RRR^{n+1}\to\RRR^{n+1+k}$ (cf. \cite{Adem:1968,Shapiro:2000}).
Note that $\rho(2)=2$ and $\rho(4)=4$.
Then $2\#2=2$ and $3\#3=4$ which follows from $4\#4=4$.
Suppose that $a\geq 2$.
Put $m=2^{a-1}$.  Since there is no immersion
$\RP^{2m}{\to}\RRR^{4m-2}$ (cf. \cite{Levine:1963}), we have $4m=(2m+1)\#(2m+1)$.
\end{proof}

Many estimations for $m\#n$ are known from immersion problem for manifolds, as projective spaces. 
For example, the existence of a nonsingular bilinear map $\RRR^{n+1}\times\RRR^{n+1}\to\RRR^{n+1+k}$ implies that $\RP^n$ immerses in $\RRR^{n+k}$ \cite{Ginsburg:1963}. 

\begin{prop}
\begin{enumerate}
\item $(n+1)\#(n+1)\leq 2n-\alpha(n)+1$ {\rm\cite{Cohen:1985}}.
\item $(2n+\alpha(n))\#(2n+\alpha(n))\geq 4n-2\alpha(n)+2$ {\rm\cite{Davis:1984}}.
\item $(8n+9)\#(8n+9)\geq 16n+6$ and $(16n+12)\#(16n+12)\geq 32n+14$ if $\alpha(n)=2$ {\rm\cite{Davis-Mahowald:2008,Singh:2004}}.
\item $(8n+10)\#(8n+10)\geq 16n+1$ and $(8n+11)\#(8n+11)\geq 16n+4$ if $\alpha(n)=3$ {\rm\cite{Davis:2011,Davis-Mahowald:2008}}.
\item $(n+1)\#(m+1)\leq n+m+1-(\alpha(n)+\alpha(n-m)+\min\{k(n),k(m)\})$ if $m$, $n$ are odd and $n\geq m$, where $k(n)$ is a nonnegative function depending only in the mod 8 residue class of $n$ with
$k(8a+1) = 0$, $k(8a+3) = k(8a+5) = 1$ and $k(8a+7) = 4$ {\rm\cite{Milgram:1967}}.
\item $d(h+1)\#(d(k-h)+\tau(k,h))\leq dk$ for $k>h\geq 0$ and $d=1,2,4,8$ {\rm\cite{Lam:1968a}}.
\item $(n+1)\# (n+\tau(2n,n))\leq 2n$.
\end{enumerate}
\end{prop}

\section{\Afcr{} tensors\mylabel{sec:afcr tensors}}

For a tensor $T$ of $\RRR^n\otimes\RRR^p\otimes\RRR^m$, we define the {\em rank} 
of $T$, denoted by $\rank T$\index{$\rank T$}, the minimal number $r$ so that there exist $\aaa_i\in\RRR^n$, $\bbb_i\in\RRR^p$, and $\ccc_i\in\RRR^m$ for $i=1, \ldots, r$ such that 
$$T=\sum_{i=1}^r\aaa_i\otimes\bbb_i\otimes\ccc_i.$$
The set $\RRR^{n\times p\times m}$ has an action of $\GL(m)\times \GL(p)\times \GL(n)$ as
$$(A,B,C)\cdot \sum_{i=1}^r \aaa_i\otimes\bbb_i\otimes\ccc_i=\sum_{i=1}^r A\aaa_i\otimes B\bbb_i\otimes C\ccc_i.
$$
For tensors $T_1,T_2\in\RRR^{m\times n\times p}$, $T_1$ and $T_2$
are said to be {\em equivalent} if
$T_1=(A,B,C)\cdot T_2$ for some $(A,B,C)\in\GL(n)\times\GL(p)\times\GL(m)$.
The equivalence relation preserves the rank.
For a subset $\msUUU$ and an open semi-algebraic subset $\msSSS$ of $\RRR^{m\times n\times p}$, we say that almost all tensors in $\msSSS$ are equivalent to tensors in $\msUUU$ if 
there exists a semi-algebraic subset $\msSSS_0$ of $\msSSS$ with $\dim \msSSS_0<mnp$ such that any tensor of $\msSSS\setminus \msSSS_0$ is equivalent to a tensor of $\msUUU$. 
In particular, for a given tensor $T_0$, if almost all tensors in $\RRR^{m\times n\times p}$ are equivalent to $\{T_0\}$, then we say that any tensor is {\em generically equivalent} to $T_0$.

An integer $r$ is called a {\em typical rank}
 of $n\times p\times m$-tensors if there is a nonempty open subset $\msOOO$ of 
$\RRR^{n\times p\times m}$ such that $\rank X=r$ for $X\in \msOOO$.
Over the complex number field $\CCC$, it is known that there is a unique typical rank, called the 
generic rank, of $n\times p\times m$-tensors for any $n$, $p$ and $m$.
The set of typical ranks of $n\times p\times m$-tensors over $\RRR$ is denoted by
$\trank(n,p,m)$ and the generic rank of $n\times p\times m$-tensors over $\CCC$ is
denoted by $\grank(n,p,m)$.

We recall the following facts.

\begin{thm}[{\cite[Theorem 7.1]{Friedland:2012}}]
\mylabel{thm:fri 7.1}
The space $\RRR^{m_1\times m_2\times m_3}$, 
$m_1$, $m_2$, $m_3\in\NNN$, 
contains a finite number of open connected
disjoint semi-algebraic sets $\msOOO_1$, \ldots, $\msOOO_M$ satisfying the following properties.
\begin{enumerate}
\item
$\RRR^{m_1\times m_2\times m_3}\setminus \bigcup_{i=1}^M\msOOO_i$
 is a closed semi-algebraic set $\RRR^{m_1\times m_2\times m_3}$ of dimension strictly less than
$m_1m_2m_3$.
\item
Each $T\in \msOOO_i$ has rank $r_i$ for $i = 1$, \ldots, $M$.
\item
${\min\{r_1,\ldots, r_M\}} = \grank(m_1,m_2,m_3)$.
\item
$\trank(m_1,m_2,m_3)=\{r\in\ZZZ \mid 
{\min\{r_1,\ldots,r_M\}}\leq r\leq {\max\{r_1,\ldots,r_M\}}\}$.
\end{enumerate}
\end{thm}

Let $T=(A_1;\ldots;A_p)$
be an $m\times n\times p$ tensor over $\RRR$.
The tensor $T$ is called an {\em \afcr} tensor if
$$
\rank(\sum_{j=1}^p y_j A_j)=n
$$
for any $(y_1,\ldots,y_p)\transpose\in\RRR^p\setminus\{\zerovec\}$.


From the definition of the \afcr\ property, we see the following fact.

\begin{lemma}
\mylabel{lem:afr stable}
Let $T$
be an $m\times n\times p$ tensor over $\RRR$ and $P\in\glin(m,\RRR)$.
Then $T$ is \afcr\ if and only if so is $PT$.
\end{lemma}

\begin{lemma}[{see Corollary \ref{cor:msa open} or \cite[Theorem~3.6]{Miyazaki-etal:2012a}}]
\mylabel{lem:afr open}
The set of $m\times n\times p$ \afcr{} tensors is an open subset of $\RRR^{m\times n\times p}$.
\end{lemma}

Let $T$ be an $m\times n\times p$-tensor.
We define $f_T\colon \RRR^n\times\RRR^p\to\RRR^m$ 
as 
$$f_T(\xxx,\yyy)=\sum_{j=1}^p y_j A_j\xxx,\index{$f_T(\xxx,\yyy)=\sum_{j=1}^p y_j A_j\xxx$}$$
where $\yyy=(y_1,\ldots,y_p)\transpose$.
Then $f_T$ is a bilinear map.
This assignment $T\mapsto f_T$ induces a bijection from $\RRR^{m\times n\times p}$ to the set of all bilinear maps $\RRR^n\times\RRR^p\to\RRR^m$.
It is easily verified that $f_T\colon \RRR^n\times\RRR^p\to\RRR^m$ is nonsingular
if and only if $T$ is \afcr.
Therefore

\begin{cor}
\mylabel{cor:afr}
There is an $m\times n\times p$ \afcr{} tensor if and only if 
there is a nonsingular bilinear map $\RRR^n\times\RRR^p\to\RRR^m$, i.e.,
$n\#p\leq m$. 
\end{cor}

\begin{lemma} \mylabel{lem:equivalent condition for afr}
Let $n$, $m$, and $u$ be positive integers with $u\leq mn$.
Set $p=mn-u$.
Then the following conditions are equivalent.
\begin{enumerate}
\item \mylabel{item:bilinear} $n\# m \leq u$.
\item \mylabel{item:afr} There is a $u\times n\times m$ \afcr{} tensor.
\item \mylabel{item:afr-sp} There is a $u\times n\times m$ \afcr{} tensor $Y$
such that
${}_{p<}\fl_1(Y)=-E_{u}$.
\end{enumerate}
\end{lemma}
\begin{proof}
\ref{item:bilinear} $\Leftrightarrow$ \ref{item:afr}
follows from Corollary \ref{cor:afr}.

It is clear that \ref{item:afr-sp} $\Rightarrow$ \ref{item:afr}.

\ref{item:afr} $\Rightarrow$ \ref{item:afr-sp}: 
Let $X=(X_1;\ldots;X_m)$ be a $u\times n\times m$ \afcr{} tensor.
By Lemma \ref{lem:afr open}, we may assume that 
${}_{p<}\fl_1(X)$ is nonsingular.
Set $Y=-{}_{p<}\fl_1(X)X$.
Then $Y$ satisfies the required conditions.
\end{proof}

\section{Ideals of minors\mylabel{sec:ideals}}
\mylabel{sec:im}

In this section, we state some results on  ideals of minors,
which we use in the following of this paper and interesting in its own right.

First we recall the definition of normality of a ring.

\begin{definition}[{see \cite[Section 9]{Matsumura:1989}}]
\rm
Let $R$ be a commutative ring.
We say that $R$ is normal if $R_P$ is an integrally closed integral domain
for any prime ideal $P$ of $R$.
\end{definition}

\begin{remark}
\rm
\mylabel{rem:normal basic}
\begin{enumerate}
\item
A Noetherian integral domain is normal if and only if it is integrally closed.
\item
If $R$ is a Noetherian normal ring, then 
$R\simeq R/P_1\times \cdots\times R/P_r$,
where $P_1$, \ldots, $P_r$ are associated prime ideals of $R$.
\end{enumerate}
\end{remark}

We recall a criterion of normality in terms of Serre's condition.

\begin{definition}[{\cite[page 183]{Matsumura:1989}}]\rm
\mylabel{def:serre cond}
Let $R$ be a Noetherian ring and $i$ a {nonnegative} integer.
\begin{enumerate}
\item
We say that $R$ satisfies \ri\ if $R_P$ is regular for any prime ideal $P$
of $R$ with $\height P\leq i$.
\item
We say that $R$ satisfies \si\ if $\depth R_P\geq\min\{i, \height P\}$
for any prime ideal $P$ of $R$.
\end{enumerate}
\end{definition}
\
\begin{lemma}[{\cite[Theorem 23.8]{Matsumura:1989}}]
\mylabel{lem:r1s2}
Let $R$ be a Noetherian ring.
Then $R$ is normal if and only if $R$ satisfies \ronestwo.
\end{lemma}

The condition \ronestwo\ is restated as follows.

\begin{lemma}
\mylabel{lem:depth1}
Let $R$ be a Noetherian ring.
Then $R$ satisfies \ronestwo\ if and only if the following condition is 
satisfied:
if $P$ is a prime ideal of $R$ with $\depth R_P\leq 1$, then $R_P$ is 
regular.
\end{lemma}

\begin{proof}
First assume that $R$ satisfies \ronestwo.
Let $P$ be a prime ideal of $R$ with $\depth R_P\leq 1$.
Since $R$ satisfies \stwo, we see that 
$\depth R_P\geq\min\{\height P, 2\}$.
Therefore, $\height P\leq 1$.
Thus by \rone, we see that $R_P$ is regular.

Conversely, assume that $R_P$ is regular for any prime ideal $P$ of
$R$ with $\depth R_P\leq 1$.
First we show that $R$ satisfies \rone.
If $P$ is a prime ideal with $\height P\leq 1$, then 
$\depth R_P\leq \height P \leq 1$.
Thus by assumption, we see that $R_P$ is regular.
Next we show that $R$ satisfies \stwo.
Let $P$ be an arbitrary prime ideal of $R$.
If $\depth R_P\leq 1$, then by assumption, $R_P$ is regular.
Thus $\depth R_P=\height P=\min\{\height P, 2\}$.
If $\depth R_P\geq 2$, then $\depth R_P\geq\min\{\height P,2\}$
holds trivially.
\end{proof}

Next we  state notations and definitions used in this section.

\begin{definition}\rm
\mylabel{def:unmtv}
We denote by $u$, $n$, $m$, and $t$ positive integers with $t\leq \min\{u,n\}$
and set 
$v=(u-t+1)(n-t+1)$.
Let $M=(m_{ij})$ be a $u\times n$ matrix with entries in a commutative ring $A$.
We denote by $I_t(M)A$, or simply $I_t(M)$, the ideal of $A$ generated by $t$-minors of $M$.
For $\alpha(1)$, \ldots, $\alpha(t)\in\{1$, \ldots, $u\}$ and
$\beta(1)$, \ldots, $\beta(t)\in\{1$, \ldots, $n\}$, we set
$[\alpha(1), \ldots, \alpha(t)\mid\beta(1),\ldots, \beta(t)]_M\define
\det(m_{\alpha(i)\beta(j)})$,
and if $u\geq n$ and 
$\alpha(1)$, \ldots, $\alpha(n)\in\{1$, \ldots, $u\}$, we set
$[\alpha(1), \ldots, \alpha(n)]_M\define\det(m_{\alpha(i)j})$.
For a tensor $T=(T_1;\ldots;T_m)$ and $\aaa=(a_1,\ldots, a_m)$
we set $M(\aaa,T)\define \sum_{i=1}^m a_iT_i$ and we define $\Gamma(u\times n)=\{[a_1,\ldots, a_n]\mid 1\leq a_1<\cdots<a_n\leq u$, $a_i\in\ZZZ\}$.
For $\gamma=[a_1, \ldots, a_n]\in\Gamma(u\times n)$,
we set $\supp\gamma=\{a_1, \ldots, a_n\}$.
If $B$ is a ring, $A$ is a subring of $B$ and $T$ is a tensor (resp.\ matrix, vector) 
with entries in $B$, we denote by $A[T]$
the subring of $B$ generated by the entries of $T$ over $A$.
If moreover, $B$ is a field, we denote by $A(T)$ the subfield of $B$ 
generated by the entries of $T$
over $A$.
If the entries of a tensor (resp. matrix, vector) $T$ are 
independent indeterminates, we say that $T$ is a tensor (resp. matrix, vector) of indeterminates.
\end{definition}

Here we note the following fact,
which is verified by using \cite[Chapter 1 Exercise 2]{Atiyah-Macdonald:1969}
 or \cite[(6.13)]{Nagata:1993}.

\begin{lemma}
\mylabel{lem:det nzd}
Let $A$ be a commutative ring, $X$ a square matrix of indeterminates.
Then $\det X$ is a non-zerodivisor of $A[X]$.
\end{lemma}

Next we recall the following fact.

\begin{lemma} [{\cite[Theorem 1 and Corollaries 3 and 4]{Hochster-Eagon:1971} 
see also \cite[(6.3) Theorem]{Bruns-Vetter:1988}}]
\mylabel{lem:det ideal}
Let $A$ be a Noetherian ring and $X$ a $u\times n$ matrix of indeterminates.
\begin{enumerate}
\item
$\height (I_t(X)A[X])=\grade (I_t(X)A[X])=v$.
\item
If $A$ is a domain, then $I_t(X)A[X]$ is a prime ideal of $A[X]$.
\item
If $A$ is a normal domain, then so is $A[X]/I_t(X)A[X]$.
\end{enumerate}
\end{lemma}

We also recall the following fact.

\begin{lemma}[{\cite[Theorem 1 and Corollaries 2 and 4]{Hochster-Eagon:1971} see also
\cite[(2.1) Theorem]{Bruns-Vetter:1988}}]
\mylabel{lem:ht max}
Let $A$ be a Noetherian commutative ring and $M$ a $u\times n$ matrix with entries in $A$.
If $I_t(M)\neq A$, then $\height I_t(M)\leq v$.
Moreover, if $A$ is \cm\  and $\height I_t(M)=v$, then $I_t(M)$ is height unmixed.
\end{lemma}

The following Lemma is a generalization of \cite[(12.4) Lemma]{Bruns-Vetter:1988}.

\begin{lemma}
\mylabel{lem:12.4+}
Let $u$, $n$, $m$, $t$ and $v$ be as in Definition \ref{def:unmtv},
$A$ a commutative Noetherian ring,
$T=(t_{ijk})$ a $u\times n\times m$ tensor of indeterminates
and $f_1$, \ldots, $f_m$ elements of $A$.
Suppose that $(f_1$, \ldots, $f_m)\neq A$.
Set $g=\grade(f_1, \ldots, f_m)A$,
$\fff=(f_1,\ldots, f_m)$ and
$M=M(\fff,T)=(m_{ij})$.
\begin{enumerate}
\item
\mylabel{item:grade}
$\grade\, I_t(M)A[T]={\min\{g,v\}}$.
\item
\mylabel{item:prime}
If $g\geq v+1$ and $A$ is a domain, then
$I_t(M)A[T]$ is a prime ideal.
\item
\mylabel{item:normal}
If $g\geq v+2$ and $A$ is a \cm\ normal domain, then  $A[T]/I_t(M)A[T]$ is
a normal domain.
\end{enumerate}
\end{lemma}

\begin{remark}\rm
If $g\geq v$, then
$\grade\, I_t(M)=\height\, I_t(M)=v$
by Lemma \ref{lem:12.4+} \ref{item:grade} and
\cite[Theorem 1 and Corollary 4]{Hochster-Eagon:1971}.
\end{remark}

\begin{proofof}{Lemma \ref{lem:12.4+}}
Set $R=A[T]$.

First we prove \ref{item:grade}.
Set $v'={\min\{g,v\}}$.
Since $I_t(M)\subset (f_1, \ldots, f_m)R$, we see that 
$\grade I_t(M)R\leq \grade(f_1,\ldots, f_m)R=g$.
Thus we see by Lemma \ref{lem:ht max}, $\grade I_t(M)R\leq v'$.

To prove the converse inequality, it is {enough} to show that 
if $P$ be a prime ideal of $R$ with $P\supset I_t(M)$,
then $\depth R_P\geq v'$.
Since if $P\supset (f_1, \ldots, f_m)R$, then $\depth R_P\geq g\geq v'$,
we may assume that $P\not \supset (f_1, \ldots, f_m)R$.
Take $l$ with $f_l\not\in P$.
Then $M$ is essentially a matrix of indeterminates over 
$A[f_l^{-1}][t_{ijk}\mid k\neq l]$.
Thus $\grade(I_t(M)R[f_l^{-1}])=v$ by Lemma \ref{lem:det ideal}.
Since $R_P$ is a localization of $R[f_l^{-1}]$, we see that $\depth R_P\geq v\geq v'$.

Next we prove \ref{item:prime}.
We may assume $f_1$, \ldots, $f_m\neq 0$.
Set $B=R/I_t(M)R$.
Since $I_t(M)R$ is grade unmixed by \ref{item:grade} and 
\cite[Theorem 1 Corollaries 2 and 4]{Hochster-Eagon:1971} (see also
\cite[Corollary of Theorem 1.2]{Rees:1957} or \cite[Exercise 16.3]{Matsumura:1989}),
we see that  every associated prime ideal of $I_t(M)R$
is of grade $v$.
In particular any associated prime ideal of $I_t(M)R$ does not
contain $(f_1, \ldots, f_m)R$,
since $g>v$ by assumption.
Thus $(\bar f_1, \ldots, \bar f_m)B$ has grade at least 1, where $\bar f_k$ denote the
natural image of $f_k$ in $B$ for $1\leq k\leq m$.

Since $A[f_l^{-1}][t_{ijk}\mid k\neq l]$ is an integral domain and $M$ is 
essentially a matrix of indeterminates over $A[f_l^{-1}][t_{ijk}\mid k\neq l]$,
we see that $B[\bar{f_l}^{-1}]=R[f_l^{-1}]/I_t(M)R[f_l^{-1}]$ 
is an integral domain for any $l$.
Thus we see that $\bar f_l$ is contained in all associated prime ideals of $B$
but one.
We denote this prime ideal by $P_l$.
Since $B[(\bar{f_l}\bar f_{l'})^{-1}]=R[(f_l f_{l'})^{-1}]/I_t(M)R[(f_lf_{l'})^{-1}]$ 
is not a zero ring by the same reason as above,
we see that
$P_l=P_{l'}$ for any $l$ and $l'$.
In particular, $P_l=P_1$ for any $l$ with $1\leq l\leq m$.
Since $\grade (\bar f_1, \ldots, \bar f_m)B\geq 1$ and any associated prime of $B$
other than $P_1$ contains $(\bar f_1, \ldots, \bar f_m)B$, we see that $P_1$ is
the only associated prime ideal of $B$.

Therefore, $B\subset B[\bar{f_1}^{-1}]$ and we see that $B$ is a domain.

Finally we prove \ref{item:normal}.
Assume that $P$ is a prime ideal of $B$ with $\depth B_P\leq 1$.
Since $B$ is \cm\ by \ref{item:grade} and
\cite[Theorem 1 and Corollary 4]{Hochster-Eagon:1971}
and $\height (\bar f_1, \ldots, \bar f_m)B\geq 2$, we see that 
$P\not\supset (\bar f_1, \ldots, \bar f_m)B$.

Take $l$ with $\bar f_l\not \in P$.
Then $B[\bar f_l^{-1}]$ is a normal domain by Lemma \ref{lem:det ideal}
and the same argument as above.
Since $B_P$ is a localization of $B[\bar f_l^{-1}]$, we see that $B_P$ is
regular by Lemmas \ref{lem:r1s2} and \ref{lem:depth1}.
Thus $B$ is normal by Lemmas \ref{lem:r1s2} and \ref{lem:depth1}.
\end{proofof}

Here we note the following fact, which can be verified by considering the
associated prime ideals of $I$ and using \cite[Theorems 15.5, 15.6]{Matsumura:1989}.

\begin{lemma}
\mylabel{lem:dim tr deg}
Let $\KKK$ be a field, $\xxx=(x_1$, \ldots, $x_m)$ a vector of indeterminates and 
$I$ a proper ideal of $\KKK[\xxx]$.
Then
$$
\dim \KKK[\xxx]/I=\max\{r\mid \exists i_1, \ldots, i_r;
\bar x_{i_1}, \ldots, \bar x_{i_r} \mbox{ are algebraically independent over $\KKK$}
\},
$$
where $\bar x_i$ denote the natural image of $x_i$ in $\KKK[\xxx]/I$.
\end{lemma}

\begin{lemma}
\mylabel{lem:elim ideal}
Let $\KKK$ be a field,
$T=(t_{ijk})$ a $u\times n\times m$-tensor of indeterminates,
and $\xxx=(x_1,\ldots,x_m)$ a vector of indeterminates.
Set $R=\KKK[T]$, $\LLL=\KKK(T)$, $M=M(\xxx,T)$ and $v'={\min\{m,v\}}$.
Then
$$
\LLL[x_1,\ldots, x_{m-v'}]\cap I_t(M)\LLL[\xxx]=(0)
,
\quad
R[x_1,\ldots, x_{m-v'}]\cap I_t(M)R[\xxx]=(0),
$$
$\LLL[\xxx]/I_t(M)\LLL[\xxx]$ is algebraic over the natural image of 
$\LLL[x_1, \ldots, x_{m-v'}]$ in $\LLL[\xxx]/I_t(M)\LLL[\xxx]$
and $R[\xxx]/I_t(M)R[\xxx]$ is algebraic over the natural image of 
$R[x_1, \ldots, x_{m-v'}]$ in $R[\xxx]/I_t(M)R[\xxx]$.
\end{lemma}

\begin{proof}
Since $I_t(M)\LLL[\xxx]$ is generated by homogeneous polynomials
of positive degree with respect to  $x_1$, \ldots, $x_m$,
we see that
$\LLL\cap I_t(M)\LLL[\xxx]=(0)$.

By Lemma~\ref{lem:12.4+}, we see that $I_t(M)$
is an ideal of height $v'$. 
Thus by {Lemma} \ref{lem:dim tr deg}, we see
$$\trdeg_\LLL\,\LLL[\xxx]/I_t(M)\LLL[\xxx]=m-v'.$$
Thus there is a permutation $i_1$, \ldots, $i_n$ of $1$, \ldots, $n$ such that
$\bar x_{i_1}$, \ldots, $\bar x_{i_{m-v'}}$ are algebraically independent over
$\LLL$ and 
$\LLL[\xxx]/I_t(M)\LLL[\xxx]$ is algebraic over 
$\LLL[\bar x_{i_1}, \ldots, \bar x_{i_{m-v'}}]$,
where $\bar x_i$ denote the natural image of $x_i$ in 
$\LLL[\xxx]/I_t(M)\LLL[\xxx]$.
By symmetry, we see that 
$\bar x_1$, \ldots, $\bar x_{m-v'}$ are algebraically independent over $\LLL$
and 
$\LLL[\xxx]/I_t(M)\LLL[\xxx]$ is algebraic over 
$\LLL[\bar x_1, \ldots, \bar x_{m-v'}]$.
We also see that 
$R[\xxx]/I_t(M)R[\xxx]$ is algebraic over 
$R[\bar x_1, \ldots, \bar x_{m-v'}]$.

Since $\bar x_1$, \ldots, $\bar x_{m-v'}$ are algebraically independent
over $\LLL$, we see that 
$$
\LLL[x_1,\ldots, x_{m-v'}]\cap I_t(M)\LLL[\xxx]=(0)
$$
and therefore
$R[x_1,\ldots, x_{m-v'}]\cap I_t(M)R[\xxx]=(0)
$.
\end{proof}

\begin{lemma}
\mylabel{lem:itm prime}
Let $\LLL/\KKK$ be a field extension with $\chara \KKK=0$,
$\xxx=(x_1, \ldots, x_m)$ a vector of {indeterminates} and $Y$
a $u\times n\times m$ tensor with entries in $\LLL$.
Set $M=M(\xxx,Y)$.
Suppose that the entries of $Y$ are algebraically independent over $\KKK$.
Then the followings hold.
\begin{enumerate}
\item
\mylabel{item:cap0}
If $m\geq v+1$, then 
$\LLL[x_1,\ldots, x_{m-v}]\cap I_t(M)\LLL[\xxx]=(0)$
and
$\LLL[\xxx]/I_t(M)\LLL[\xxx]$ is algebraic over the natural image of $\LLL[x_1, \ldots, x_{m-v}]$.
\item
\mylabel{item:itmprime}
If $m\geq v+2$, then 
$\LLL[\xxx]/I_t(M)\LLL[\xxx]$ is a normal domain.
In particular, $I_t(M)\LLL[\xxx]$ is a prime ideal of $\LLL[\xxx]$ of 
height $v$.
\end{enumerate}
\end{lemma}

\begin{proof}
Since the entries of $Y$ are algebraically independent over $\KKK$,
we see by Lemma \ref{lem:elim ideal} that
$\KKK(Y)[\xxx]/I_t(M)\KKK(Y)[\xxx]$ is algebraic over 
$\KKK(Y)[x_1, \ldots, x_{m-v}]$.
Thus 
$\LLL[\xxx]/I_t(M)\LLL[\xxx]$ is algebraic over 
$\LLL[x_1, \ldots, x_{m-v}]$ since 
$\LLL[\xxx]/I_t(M)\LLL[\xxx]=
(\KKK(Y)[\xxx]/I_t(M)\KKK(Y)[\xxx])\otimes_{\KKK(Y)}\LLL$.
On the other hand, since 
$\trdeg_\LLL \LLL[\xxx]/I_t(M)\LLL[\xxx]
=\dim\LLL[\xxx]/I_t(M)\LLL[\xxx]
\geq m-v$,
by Lemmas \ref{lem:ht max} and \ref{lem:dim tr deg}, we see that 
$\bar x_1$, \ldots, $\bar x_{m-v}$ are algebraically independent over
$\LLL$, where $\bar x_i$ denote the natural image of $x_i$ in 
$\LLL[\xxx]/I_t(M)\LLL[\xxx]$.
Thus
$
\LLL[x_1, \ldots, x_{m-v}]\cap I_t(M)\LLL[\xxx]=(0)
$.
This proves \ref{item:cap0}.

Next we prove \ref{item:itmprime}.
Take a transcendence basis $S$ of $\LLL/\KKK(Y)$
and put
\begin{eqnarray*}
A&=&\KKK(Y)[\xxx]/I_t(M)\KKK(Y)[\xxx],\\
C&=&\KKK(Y)(S)[\xxx]/I_t(M)\KKK(Y)(S)[\xxx] \mbox{ and }\\
B&=&\LLL[\xxx]/I_t(M)\LLL[\xxx].
\end{eqnarray*}
By Lemma \ref{lem:12.4+} \ref{item:normal}, we see that $A$ is a normal domain.

Since
$$
\KKK(Y)[S][\xxx]/I_t(M)\KKK(Y)[S][\xxx]=
(\KKK(Y)[\xxx]/I_t(M)\KKK(Y)[\xxx])\otimes_{\KKK(Y)}\KKK(Y)[S]=
A[S]
$$
is a polynomial ring (with possibly infinitely many variables) over $A$,
it is an integrally closed integral domain.
Since $C$ is a localization of the above ring, $C$ is a normal domain.

Now let $P$ be a prime ideal of $B$ with $\depth B_P\leq 1$.
By Lemmas \ref{lem:r1s2} and \ref{lem:depth1}, 
it is enough to show that $B_P$ is regular.
Set $Q=C\cap P$.
Then since $B=C\otimes_{\KKK(Y)(S)}\LLL$ is flat over $C$, we see
that $\depth C_Q\leq 1$ by \cite[Theorem 23.3 Corollary]{Matsumura:1989}.
Thus $C_Q$ is regular since $C$ is normal.
The fiber ring $C_Q/QC_Q\otimes_{C}B_P$ is a localization of
$$
C_Q/QC_Q\otimes_{C}B=C_Q/QC_Q\otimes_{\KKK(Y)(S)}\LLL
$$
which is a 0-dimensional reduced ring, thus regular,
since $\LLL$ is {separably} algebraic over $\KKK(Y)(S)$.
(Note that $\LLL$ is an inductive limit of finitely generated algebraic
extension fields of $\KKK(Y)(S)$.
Or see \cite[Theorem 3.2.6 and Theorem 3.2.8 (i)]{Nagata:1993}
and note the assumption of the existence of a field containing $M$ and $N$ is
not used in the proof of \cite[Theorem 3.2.8 (i)]{Nagata:1993}.)
Thus by \cite[Theorem 23.7]{Matsumura:1989}, we see that $B_P$ is regular.

Thus, $B$ is a normal ring.
Since $B$ is a nonnegatively graded ring whose degree 0 component is a field,
$B$ is not a direct product of 2 or more rings.
Therefore, $B$ is a domain by Remark \ref{rem:normal basic}.
Moreover, we see by \ref{item:cap0}, that $\height I_t(M)\LLL[\xxx]=v$.
\end{proof}

\begin{definition}\rm
\mylabel{def:mu aind psi}
Let $u$, $n$, $m$, $t$ and $v$ be as in Definition \ref{def:unmtv}.
We set
$$\begin{array}{l}
\msAAA_t^{u\times n\times m}\define
\{Y\in \RRR^{u\times n\times m}\mid I_t(M(\aaa,Y))\neq(0)\mbox{ for any }
\aaa\in\RRR^{1\times m}\setminus\{\zerovec\} \}, \\
\msCCC_t^{u\times n\times m}\define \RRR^{u\times n\times m}\setminus
\msAAA_t^{u\times n\times m}, \\
\Aind\define\{Y\in\RRR^{u\times n\times m} \mid \mbox{the entries of $Y$
 are algebraically independent over $\QQQ$}\},
\end{array}
$$
and
for $Y=(Y_1;\ldots;Y_m)\in\RRR^{u\times n\times m}$ and for integers
 $k$, $k^\prime$ with $t\leq k\leq u$ and $t\leq k^\prime\leq n$, we set
$$
\mu_{k,k^\prime}(\xxx,Y)\define[1,\ldots,t-1,k\mid 1,\ldots,t-1,k^\prime]_{M(\xxx,Y)},
$$
where $\xxx$ is a vector of indeterminates.
We also define
$$
\begin{array}{l}
J_t(\xxx,Y)=\displaystyle\frac{\partial(\mu_{tt}, \mu_{t,t+1},\ldots, \mu_{tn}, \mu_{t+1,t}, \ldots, \mu_{t+1,n},\ldots,\mu_{u,t},\ldots,\mu_{un})}
{\partial(x_{m-v+1}, \ldots, x_m)}(\xxx,Y), \\
S_t(Y)\define \left\{\aaa\in\RRR^{1\times m} \mid \begin{array}{l}
\det M(\aaa,Y)^{<t}_{<t}\neq 0, \ 
J_t(\aaa,Y)\neq 0 
\text{ and } I_t(M(\aaa,Y))=(0) \end{array}\right\}, \\
\msPPP_t\define\{Y\in\RRR^{u\times n\times m}\mid S_t(Y)\neq\emptyset \}.
\end{array}$$
\end{definition}

\begin{remark}\rm
\mylabel{rem:char pt}
\begin{enumerate}
\item
$
\msAAA_1^{u\times n\times m}\supset\msAAA_2^{u\times n\times m}
\supset\cdots\supset\msAAA_{{\min\{u,n\}}}^{u\times n\times m}
$.
\item
$\msAAA_t^{u\times n\times m}$ is stabel under the action of $\glin(u,\RRR)$
 for any $t$.
\item
$\msCCC_t^{u\times n\times m}\neq\emptyset$
for any $t$.
\item
$\msPPP_t$ is a subset of $\msCCC^{u\times n\times m}_t$ and
$$
S_t(Y)\define \left\{\aaa\in\RRR^{1\times m} \left| \begin{array}{l}
\det M(\aaa,Y)^{<t}_{<t}\neq 0, \ 
J_t(\aaa,Y)\neq 0 \\
\text{and there exist linearly independent vectors } \bbb_t, \ldots, \bbb_n\in \RRR^n  \\
\text{such that }
M(\aaa,Y)\bbb_j=\zerovec
\text{ for any $t\leq j\leq n$} \end{array}\right.\right\}
$$
\end{enumerate}
\end{remark}

\begin{lemma}
\mylabel{lem:itm0}
Let $u$, $n$ and $t$ be as in Definition \ref{def:unmtv},
$A$ an integral domain and
$M$ a $u\times n$ matrix with entries in $A$.
Suppose that $\det M^{<t}_{<t}\neq 0$
and
$[1, \ldots, t-1,k\mid 1,\ldots, t-1,k']_M=0$ for any integer with
$t\leq k\leq u$ and $t\leq k'\leq n$.
Then $I_t(M)=(0)$.
In particular,
$$
S_t(Y)\define \left\{\aaa\in\RRR^{1\times m} \left| \begin{array}{l}
\det M(\aaa,Y)^{<t}_{<t}\neq 0, \ 
J_t(\aaa,Y)\neq 0 \\
\text{and $[1, \ldots, t-1,k\mid 1,\ldots, t-1,k']_M=0$ for any integer with} \\
\text{$t\leq k\leq u$ and $t\leq k'\leq n$.} \\
\end{array}\right.\right\}
$$
\end{lemma}
\begin{proof}
Set
$$
\xi_l\define
\begin{pmatrix}
(-1)^{t+1}[1, \ldots, t-1\mid 2, \ldots, t-1, l]_{M}\\
(-1)^{t+2}[1, \ldots, t-1\mid 1,3, \ldots, t-1, l]_{M}\\
\vdots\\
(-1)^{2t-1}[1, \ldots, t-1\mid 1, \ldots, t-2, l]_{M}\\
\zerovec \\
(-1)^{2t}[1, \ldots, t-1\mid 1, \ldots, t-2, t-1]_{M}\\
\zerovec
\end{pmatrix}\in A^n
$$
for each $l$ with $t\leq l\leq n$
($(-1)^{2t}[1, \ldots, t-1\mid 1, \ldots, t-2, t-1]_{M}$ in the $l$-th position).
Then, since $[1,\ldots, t-1\mid 1\ldots, t-1]_M=\det M^{<t}_{<t}\neq 0$,
we see that 
$\xi_t$, \ldots, $\xi_n$ are linearly independent over $A$.
Since the $k$-th entry of $M\xi_l$ is $[1, \ldots, t-1,k\mid1,\ldots, t-1,l]_M$,
we see, by assumption, that
$M\xi_l=\zerovec$ for $t\leq l\leq n$.
Thus, $\rank M< t$ and we see that $I_t(M)=(0)$.
\end{proof}

It is verified the following fact, since $\QQQ$ is a countable field.

\begin{lemma}
\mylabel{lem:aind dense}
$\Aind$ is a dense subset of $\RRR^{u\times n\times m}$.
\end{lemma}

We also see that $\msAAA_t^{u\times n\times m}$ is an open subset of $\RRR^{u\times n\times m}$.
First note the following fact, which is easily verified.

\begin{lemma}
\mylabel{lem:min cont}
Let $X$ and $Y$ be topological spaces with $X$ compact and
$f\colon X\times Y \to \RRR$ is a continuous map.
Set $g\colon Y\to \RRR$ by 
$g(y)\define \min_{x\in X} f(x,y)$.
Then, $g$ is a continuous map.
\end{lemma}

\begin{cor}
\mylabel{cor:msa open}
$\msAAA_t^{u\times n\times m}$ is an open subset of $\RRR^{u\times n\times m}$.
\end{cor}

\begin{proof}
Since $\msAAA_t^{u\times n\times m}$ is the set consisting of
$Y\in \RRR^{u\times n\times m}$ such that
$$\min_{\aaa\in S^{m-1}}(\text{the maximum of the absolute values of 
$t$-minors of } M(a,Y))>0,$$
we see the result by the previous lemma.
\end{proof}

\begin{lemma}
\mylabel{lem:o1 dense}
If $v<m$, then
$\msPPP_t$ is a dense subset of $\msCCC_t^{u\times n\times m}$.
In particular, $\msPPP_t\neq\emptyset$.
\end{lemma}

\begin{proof}
Let $Y\in\msCCC_t^{u\times n\times m}$ and $\msUUU$ 
an open neighborhood of $Y$ in $\RRR^{u\times n\times m}$. 
In order to prove the first assertion,
it suffices to show that $\msPPP_t\cap \msUUU\neq\emptyset$.

There exist a nonzero vector $\aaa\in\RRR^{1\times m}$ and linearly independent vectors $\bbb_t$, \ldots, $\bbb_n\in\RRR^n$ such that 
$M(\aaa,Y)\bbb_j=\zerovec$ for $t\leq j\leq n$.
Let $g_3\in \GL(m)$ and $g_2\in \GL(n)$ such that the first entry of $g_3\transpose\aaa$ is nonzero, 
${}^{t\leq}(g_2\transpose\bbb_t, \ldots, g_2\transpose\bbb_n)$ is nonsingular
and sufficiently close to $E_m$ and $E_n$ respectively
so that $(1,g_2^{-1},g_3^{-1})\cdot Y\in\msUUU$.
By replacing $Y$, $\aaa$  and $\bbb_t$, \ldots, $\bbb_n$ by $(1,g_2^{-1},g_3^{-1})\cdot Y$, 
$g_3\transpose\aaa$
and $g_2\transpose\bbb_t$, \ldots, $g_2\transpose\bbb_n$ respectively,  
we may assume that the first entry of $\aaa$ is nonzero and
${}^{t\leq}(\bbb_t, \ldots, \bbb_n)$ is nonsingular.

Let $e\in\RRR$.
We take a tensor $P(e)=(p_{ijk})\in\RRR^{u\times n\times m}$ as follows.
For any $i$, $j$, $k$ with $j<t$ or $k\neq 1$, we put
$$p_{ijk}=\begin{cases}
e^{ij} & (k=1, i<t, j<t,) \\
e & ((i,j,k)=(t+l,t+l^\prime,m-v+1+l+l^\prime(u-t+1)),\\
&\hskip 10mm 0\leq l\leq u-t, 0\leq l^\prime\leq n-t) \\
0 & (\text{otherwise}) \\
\end{cases}$$
and take $p_{ij1}$ for $1\leq i\leq u$ and $t\leq j\leq n$ so that $M(\aaa,P(e))\bbb_j=\zerovec$
for $t\leq j\leq n$.
Note that we can take such $p_{ij1}$ since the first entry of $\aaa$ is nonzero and
${}^{t\leq}(\bbb_t, \ldots, \bbb_n)$ is nonsingular.

Then we have
$$
\det M(\aaa, Y+P(e))^{<t}_{<t}\neq 0
\quad
\mbox{and}
\quad
J_t(\aaa,Y+P(e))\neq 0
$$
for $e\gg 0$.

Therefore, 
since the entries of $P(e)$ are polynomials of $e$, we see that
for a real number $e_0\neq 0$ which is sufficiently closed to $0$, 
\begin{eqnarray}
\mylabel{eq:cd3}
\det M(\aaa, Y+P(e_0))^{<n}_{<n}\neq 0, \\
\mylabel{eq:cd4}
J_t(\aaa,Y+P(e_0))\neq 0, \\
Y+P(e_0)\in \msUUU.
\end{eqnarray}
\eqref{eq:cd3}, \eqref{eq:cd4} and the fact 
$M(\aaa, Y+P(e_0))\bbb_j=M(\aaa,Y)\bbb_j+M(\aaa,P(e_0))\bbb_j=\zerovec$
for $t\leq j\leq n$ imply that $\aaa\in S(Y+P(e_0))$.
Thus we have $Y+P(e_0)\in\msPPP_t$ and we see that $\msPPP_t\cap\msUUU\neq\emptyset$.

The latter assertion follows from Remark \ref{rem:char pt}.
\end{proof}

\begin{lemma}
\mylabel{lem:imp func}
Suppose that $v<m$.
Then the set $\msPPP_t$ is an open subset of $\RRR^{u\times n\times m}$ and
for any $Y\in \msPPP_t$ and 
$\aaa=(\aaa_1,\aaa_2)\in S_t(Y)$, where
$\aaa_1\in \RRR^{1\times (m-v)}$ and $\aaa_2\in\RRR^{1\times v}$,
there exists an open neighborhood $O(\aaa,Y)$ of $\aaa_1\in\RRR^{1\times (m-v)}$
such that
for any $\bbb_1\in O(\aaa,Y)$, there exists
$\bbb_2\in \RRR^{1\times v}$ such that $(\bbb_1,\bbb_2)\in S_t(Y)$.
\end{lemma}

\begin{proof}
Assume that $Y\in \msPPP_t$ and 
$\aaa=(\aaa_1,\aaa_2)\in S_t(Y)$.
Then $\mu_{kk^\prime}(\aaa,Y)=0$ for any $t\leq k\leq u$ and $t\leq k^\prime\leq n$.
Thus by implicit function theorem, we see that there is an open neighborhood
$U$ of $(\aaa_1,Y)$ in $\RRR^{1\times (m-v)}\times\RRR^{u\times n\times m}$ 
and a continuous map $\nu \colon U\to \RRR^{1\times v}$ such that
$\nu(\aaa_1,Y)=\aaa_2$, and
$\mu_{kk^\prime}(\bbb,Z)=0$ for any 
$(\bbb_1,Z)\in U$ and any $k$, $k^\prime$ with $t\leq k\leq u$ and $t\leq k^\prime\leq n$, where 
$\bbb\define(\bbb_1, \nu(\bbb_1,Z))$.
By replacing $U$ by a smaller neighborhood if necessary, we may assume that
$\det M(\bbb,Z)^{<t}_{<t}\neq 0$
and
$J_t(\bbb,Z)\neq 0$ for any $(\bbb_1,Z)\in U$.

Assume $(\bbb_1,Z)\in U$. 
Put 
$\bbb=(\bbb_1, \nu(\bbb_1,Z))$.
By Lemma \ref{lem:itm0}, we see that $\bbb\in S_t(Z)$.
Thus it suffices to set $O(\aaa,Y)\define\{\bbb_1\in\RRR^{1\times (m-v)}\mid (\bbb_1,Y)\in U\}$.
Moreover, since $\{Z\in\RRR^{u\times n\times m}\mid (\aaa_1,Z)\in U\}$ is an
open subset of $\msPPP_t$ containing $Y$, 
we see that $\msPPP_t$ is an open subset of $\RRR^{u\times n\times m}$.
\end{proof}

By Corollary~\ref{cor:msa open} and Lemmas \ref{lem:o1 dense} and \ref{lem:imp func}, we see the following:

\begin{cor}
\mylabel{cor:int closure}
If $v<m$, then
$\msPPP_t\subset\interior\msCCC_t^{u\times n\times m}$ and
$\overline\msPPP_t=
\overline{\interior \msCCC_t^{u\times n\times m}}=\msCCC_t^{u\times n\times m}$.
\end{cor}


\begin{definition}\rm
\mylabel{def:tilde p}
We set
$\tilde\msPPP_t\define\{PY\mid P\in\glin(u,\RRR), Y\in\msPPP_t\}$.
\end{definition}

\begin{lemma}
\mylabel{lem:tilde p open}
$\tilde\msPPP_t$ is an open subset of $\RRR^{u\times n\times m}$,
stable under the action of $\glin(u,\RRR)$ 
and $\overline{\tilde\msPPP_t}=\msCCC_t^{u\times n\times m}$.
\end{lemma}
\begin{proof}
Since $\tilde \msPPP_t=\bigcup_{P\in\glin(u,\RRR)}P\msPPP_t$ 
and $P\msPPP_t$ is an open subset of $\RRR^{u\times n\times m}$
for any $P\in\glin(u,\RRR)$ by Lemma \ref{lem:imp func} and the fact that
multiplication of a nonsingular matrix is a homeomorphism on $\RRR^{u\times n\times m}$.
Therefore, $\tilde\msPPP_t$ is an open subset of $\RRR^{u\times n\times m}$.
The fact that ${\tilde\msPPP_t}$ is stable under the action of $\glin(u,\RRR)$ 
is clear from the definition of $\tilde \msPPP_t$.
Finally, since $\msAAA_t^{u\times n\times m}$ is stable under the action
of $\glin(u,\RRR)$,
we see,
 by Remark \ref{rem:char pt},
that $\tilde\msPPP_t\subset\msCCC_t^{u\times n\times m}$. 
Therefore, we see that 
$\overline{\tilde\msPPP_t}=\msCCC_t^{u\times n\times m}$
by Lemma \ref{lem:o1 dense}.
\end{proof}

\begin{lemma}
\mylabel{lem:ht max open}
Let $\LLL$ be an infinite field
and $\xxx=(x_1, \ldots, x_m)$ a vector of indeterminates.
Set $v'={\min\{m,v\}}$ and $v''={\min\{m,(u-t+2)(n-t+2)\}}$.
Then there is a Zariski dense open subset 
$\msQQQ_1$ 
of 
$\LLL^{u\times n\times m}$ such that if 
$Y\in\msQQQ_1$,
then $\LLL[\xxx]/I_t(M(\xxx,Y))\LLL[\xxx]$ is algebraic over the natural image of
$\LLL[x_1, \ldots, x_{m-v'}]$,
$\LLL[x_1, \ldots, x_{m-v'}]\cap I_t(M(\xxx,Y))\LLL[\xxx]=(0)$,
$\height I_t(M)\LLL[\xxx]=v'$,
$\LLL[\xxx]/I_{t-1}(M(\xxx,Y))\LLL[\xxx]$ is algebraic over the natural image of
$\LLL[x_1, \ldots, x_{m-v''}]$,
$\LLL[x_1, \ldots, x_{m-v''}]\cap I_{t-1}(M(\xxx,Y))\LLL[\xxx]=(0)$
and
$\height I_{t-1}(M(\xxx,Y))\LLL[\xxx]=v''$.
\end{lemma}
\begin{proof}
Let $T=(t_{ijk})$ be the $u\times n\times m$ tensor of indeterminates.
Then by Lemma
\ref{lem:elim ideal}, we see that 
$\LLL[T][\xxx]/I_t(M(\xxx,T))\LLL[T][\xxx]$ is
algebraic over the natural image of $\LLL[T][x_1, \ldots, x_{m-v'}]$.
Denote the natural image of $x_l$  in $\LLL[T][\xxx]/I_t(M(\xxx,T))\LLL[T][\xxx]$ by $\bar x_l$.
Take a nonzero polynomial $f_l(t)$ with coefficient in $\LLL[T][x_1, \ldots, x_{m-v'}]$ 
such that $f_l(\bar x_l)=0$ for each $l$ with $m-v'+1\leq l\leq m$.
Let $g$ be the product of all nonzero elements of $\LLL[T]$ appearing 
as 
the coefficient of at least one of $f_l$ and set 
$\msQQQ'_1=\LLL^{u\times n\times m}\setminus\VVV(g)$.
Then 
$\msQQQ'_1$ is a Zariski dense open subset of $\LLL^{u\times n\times m}$.

Suppose that $Y\in\msQQQ'_1$.
And let $\tilde x_i$ be the natural image of $x_i$ in $\LLL[\xxx]/I_t(M(\xxx,Y))\LLL[\xxx]$
and $\tilde f_l$ be an element of $\LLL[x_1, \ldots, x_{m-v'}]$
obtained by substituting $Y$ to $T$.
Then $\tilde f_l$ is a nonzero element of $\LLL[x_1, \ldots, x_{m-v'}]$ and
$\tilde f_l(\tilde x_l)=0$ for $m-v'+1\leq l\leq m$.
Therefore, $\LLL[\xxx]/I_t(M(\xxx,Y))\LLL[\xxx]$ is algebraic over 
the natural image of $\LLL[x_1,\ldots, x_{m-v'}]$.
Thus $\height I_t(M(\xxx,Y))\LLL[\xxx]
=m-\dim\LLL[\xxx]/I_t(M)\LLL[\xxx]
=m-\trdeg_\LLL \LLL[\xxx]/I_t(M)\LLL[\xxx]
\geq v'$.
Thus, $\height I_t(M(\xxx,Y))\LLL[\xxx]=v'$ by Lemma \ref{lem:ht max} and we see that
$\trdeg_\LLL \LLL[\xxx]/I_t(M(\xxx,Y))\LLL[\xxx]
= m-v'$.
Therefore
$\tilde x_1$, \ldots, $\tilde x_{m-v'}$ are algebraically independent over $\LLL$,
that is,
$\LLL[x_1,\ldots, x_{m-v'}]\cap I_t(M(\xxx,Y))\LLL[\xxx]=(0)$.

We see by the same way that there is a Zariski dense open subset $\msQQQ''_1$ of
$\LLL^{u\times n\times m}$ such that if $Y\in\msQQQ''_1$, then
$\LLL[\xxx]/I_{t-1}(M(\xxx,Y))\LLL[\xxx]$ is algebraic over the natural image of
$\LLL[x_1, \ldots, x_{m-v''}]$,
$\LLL[x_1, \ldots, x_{m-v''}]\cap I_{t-1}(M(\xxx,Y))\LLL[\xxx]=(0)$
and
$\height I_{t-1}(M(\xxx,Y))\LLL[\xxx]=v''$.
Thus it is enough to set $\msQQQ_1=\msQQQ'_1\cap\msQQQ''_1$.
\end{proof}

Let $\LLL$ be a field, $\xxx=(x_1, \ldots, x_m)$
a vector of indeterminates and $M$ a $u\times n$ matrix with 
entries in $\LLL[\xxx]$.
Suppose that $\height I_t(M)=v$ and $\det(M^{<t}_{<t})\not\in\sqrt{I_t(M)}$.
Then $I_t(M)\LLL[\xxx][\det(M^{<t}_{<t})^{-1}]$
is a proper ideal of $\LLL[\xxx][\det(M^{<t}_{<t})^{-1}]$
and $M$ is
equivalent to the matrix of the following form in
$\LLL[\xxx][(\det(M^{<t}_{<t})^{-1}]$.
$$
\begin{pmatrix}E_{t-1}&O\\O&\ast\end{pmatrix}
$$
In particular, $I_t(M)$ is a complete intersection ideal in
$\LLL[\xxx][(\det(M^{<t}_{<t})^{-1}]$.
By symmetry, we see that if $\height I_{t-1}(M)>v$, then
$I_t(M)$ is a generically complete intersection ideal.

We use the notation of \cite[p.\ 219]{Eisenbud-etal:1992}.
Let $\LLL$ be a field of characteristic 0, $T$ a $u\times n\times m$ tensor of indeterminates 
and $\xxx=(x_1, \ldots, x_m)$ a vector of indeterminates.
Set $\MMM=\LLL(T)$. 
Suppose that $m>v$.
Then 
$I_t(M(\xxx,T))\MMM[\xxx]$ is a prime ideal and 
$\height I_{t-1}(M(\xxx,T))\MMM[\xxx]>v$ by
Lemma \ref{lem:12.4+}.
Thus
$$
I_t(M(\xxx,T)):\msJJJ_{m-v}(I_t(M(\xxx,T)))=I_t(M(\xxx,T))
$$
by \cite[Theorem 2.1]{Eisenbud-etal:1992}
and the argument above.
Thus if we set $I'=I_t(M(\xxx,T))+\msJJJ_{m-v}(I_t(M(\xxx,T)))$,
then $\height I'>v$.
Therefore the natural images of $x_1$, \ldots, $x_{m-v}$ 
in $\MMM[\xxx]/I'$ are algebraically dependent
over $\MMM$ 
by Lemma \ref{lem:dim tr deg}.
Take a transcendence basis $x_{i_1}$, \ldots $x_{i_d}$ of
$\MMM[\xxx]/I'$ over $\MMM$.
By symmetry, we may assume that $i_k=k$ for $1\leq k\leq d$.
Since $\height I'>v$, we see that $d<m-v$.
Take a nonzero polynomial $f(t)$ with coefficients in $\LLL[T][x_1, \ldots, x_d]$
such that $f(x_{m-v})\in I'$
and let $g$ be the product of all nonzero elements of $\LLL[T]$
which appear in some nonzero {coefficient} of $f$.
Set $\msQQQ_2=\LLL^{u\times n\times m}\setminus \VVV(g)$.
Then $\msQQQ_2$ is a Zariski dense open subset 
of $\LLL^{u\times n\times m}$
and if $Y\in\msQQQ_1\cap\msQQQ_2$, where $\msQQQ_1$ is the one
in Lemma \ref{lem:ht max open}, then
$\height (I_t(M(\xxx,Y))+\msJJJ_{m-v}(I_t(M(\xxx,Y))))>v$
since 
$\trdeg_{\LLL}\LLL[\xxx]/(I_{t}(M(\xxx,Y))+\msJJJ_{m-v}(I_t(M(\xxx,Y))))<m-v$.

\medskip

Until the end of this section, assume that $m\geq v+2$ and 
let $U$ be the $m\times m$ matrix of indeterminates,
$T$ the $u\times n\times m$ tensor of indeterminates,
$\xxx=(x_1, \ldots, x_m)$ the vector of indeterminates
and $\LLL$ the algebraic closure of $\RRR(U)$.

Set
$$
\begin{pmatrix}x_1'\\\vdots\\x_m'\end{pmatrix}
=U
\begin{pmatrix}x_1\\\vdots\\x_m\end{pmatrix}.
$$
Then $\LLL(T)[x_1',\ldots, x_m']=\LLL(T)[\xxx]$
and
$\LLL(T)[x_1',\ldots, x_{m-v+1}']\cap I_t(M(\xxx,T))$
is a principal ideal 
generated by a polynomial $F$ called the ground form of $I_t(M(\xxx,T))$,
since $I_t(M(\xxx,T))$ is a prime ideal therefore is unmixed of height $v$.
See
\cite[parts 28 and 29]{Seidenberg:1974}.

Since $\LLL(T)[x_1',\ldots,x_{m-v+1}']\cap I_t(M(\xxx,T))$
is the elimination ideal, $F$ is obtained by the Buchberger's algorithm.
Let $g_3$ be the products of all elements of $\LLL[T]$ which {appear} as a 
numerator or a denominator of a
nonzero
coefficient of at least one polynomial in the process of Buchberger's algorithm
to obtain the reduced Gr\"obner basis of $I_t(M(\xxx,T))$ in $\LLL[T][\xxx]$.
Set $\msQQQ_3=\LLL^{u\times n\times m}\setminus\VVV(g_3)$.
Then $\msQQQ_3$ is a Zariski dense open subset of $\LLL^{u\times n\times m}$
and if $Y\in\msQQQ_3$, 
then the Buchberger's algorithm to obtain the reduced Gr\"obner basis of 
$I_t(M(\xxx,Y))\LLL[\xxx]$ in $\LLL[\xxx]$ is identical with that of
$I_t(M(\xxx,T))\LLL(T)[\xxx]$ in $\LLL(T)[\xxx]$.
In particular,
$\LLL[x_1',\ldots, x_{m-v+1}']\cap I_t(M(\xxx,Y))$ is a principal ideal
generated by $F_Y$, the polynomial obtained by substituting $Y$ in $T$ in
the coefficients of $F$.

Let $d=\deg F$ and  let $P_\LLL(d,m-v+1)$ (resp. $P_\RRR(d,m-v+1)$) 
be the set of homogeneous polynomials
with coefficients in $\LLL$ (resp. $\RRR$) with variables $x_1'$, \ldots, $x_{m-v+1}'$ and
degree $d$.
Since $m-v+1\geq 3$ and $\LLL$ is an algebraically closed field containing $\RRR$,
we see by \cite{Heintz-Sieveking:1981} that
\begin{eqnarray*}
&&
\{G\in P_\RRR(d,m-v+1)\mid \mbox{$G$ is absolutely irreducible}\}\\
&=&P_\RRR(d,m-v+1)\cap
\{G\in P_\LLL(d,m-v+1)\mid \mbox{$G$ is irreducible}\}
\end{eqnarray*}
is a Zariski dense open subset of $P_\RRR(d,m-v+1)$.

\begin{definition}\rm
\mylabel{def:msqqq}
Set 
$$
\msQQQ=\{Y\in\msQQQ_1\cap\msQQQ_2\cap\msQQQ_3\cap\RRR^{u\times n\times m}\mid
\mbox{$F_Y$ is absolutely irreducible}\},
$$
where $\msQQQ_1$ is the one in Lemma \ref{lem:ht max open} and
$\msQQQ_2$, $\msQQQ_3$ and $F_Y$ are the ones defined after the proof
of Lemma \ref{lem:ht max open}.
\end{definition}

\begin{remark}\rm
$\msQQQ$ is a Zariski open subset of $\RRR^{u\times n\times m}$,
since 
the correspondence $Y$ to $F_Y$ is a rational map whose domain contains $\msQQQ_3$.
\end{remark}

Moreover, we see the following fact.

\begin{lemma}
$\msQQQ\supset\Aind\cap\msQQQ_1\cap\msQQQ_2\cap\msQQQ_3$.
In particular, $\msQQQ$ is not an empty set.
\end{lemma}
\begin{proof}
Suppose that $Y\in \Aind\cap\msQQQ_1\cap\msQQQ_2\cap\msQQQ_3$.
Then we see, by applying Lemma \ref{lem:itm prime} to $\LLL/\QQQ$, 
that
$I_t(M(\xxx,Y))\LLL[\xxx]$ is a prime ideal.
Thus the elimination ideal is also prime and therefore the 
generator
$F_Y$ of the elimination ideal is an irreducible polynomial in $\LLL[x_1',\ldots, x_{m-v+1}']$.
Therefore, $Y\in\msQQQ$.
Since $\msIII$  is a dense subset of $\RRR^{u\times n\times m}$ by Lemma \ref{lem:aind dense}, 
we see that $\msIII\cap\msQQQ_1\cap\msQQQ_2\cap\msQQQ_3\neq\emptyset$.
Thus, $\msQQQ\neq \emptyset$.
\end{proof}

Thus we see that $\msQQQ$ is a non-empty Zariski open subset of $\RRR^{u\times n\times m}$.
In particular, $\msQQQ$ is dense.

\begin{lemma}
\mylabel{lem:itmy prime}
If $Y\in\msQQQ$, then $I_t(M(\xxx,Y))\RRR[\xxx]$ is a prime ideal of height $v$.
\end{lemma}
\begin{proof}
Since $Y\in\msQQQ$, 
$\height I_t(M(\xxx,Y))\RRR[\xxx]=v$ and
$\LLL[x'_1, \ldots, x'_{m-v+1}]\cap I_t(M(\xxx,Y))\LLL[\xxx]=
(F_Y)\LLL[x'_1, \ldots, x'_{m-v+1}]$.
Thus
$\RRR(U)[x_1',\ldots, x_{m-v+1}']\cap I_t(M(\xxx,Y))\RRR(U)[\xxx]
=(F_Y)\RRR(U)[x'_1, \ldots, x'_{m-v+1}]$
since $\LLL$ is faithfully flat over $\RRR(U)$.
Thus we see that $F_Y$ is the ground form of $I_t(M(\xxx,Y))\RRR[\xxx]$
\SeidenbergCite{28}.
Since
$F_Y$ is an irreducible polynomial in 
$\LLL[x'_1, \ldots, x'_{m-v+1}]$ and therefore in
$\RRR(U)[x_1',\ldots, x_{m-v+1}']$,
we see by \SeidenbergCite{31}, that $I_t(M(\xxx,Y))\RRR[\xxx]$ 
is a primary ideal.

On the other hand, since $Y\in\msQQQ_1\cap\msQQQ_2$, we see that
$\height (I_t(M(\xxx,Y))+\msJJJ_{m-v}(I_t(M(\xxx,Y))))>v$.
Since $I_t(M(\xxx,Y))$ is a primary ideal of height $v$,
we see that
$$
I_t(M(\xxx,Y)):\msJJJ_{m-v}(I_t(M(\xxx,Y)))=I_t(M(\xxx,Y)).
$$
Therefore, by \cite[Theorem 2.1]{Eisenbud-etal:1992}, we see that
$I_t(M(\xxx,Y))$ is a radical ideal.
Thus $I_t(M(\xxx,Y))$ is a prime ideal.
\end{proof}

Now we show the following result.

\begin{thm}
\mylabel{thm:it real}
Suppose that $m\geq v+2$.
Set $\msOOO_1=\msQQQ\cap\tilde\msPPP_t$ and $\msOOO_2=\msQQQ\cap\msAAA_t^{{u\times n\times m}}$\index{$\msOOO_2=\msQQQ\cap\msAAA_t^{u\times n\times m}$}.
Then the followings hold.
\begin{enumerate}
\item
\mylabel{item:msooo open}
$\msOOO_1$ and $\msOOO_2$ are disjoint open subsets of $\RRR^{u\times n\times m}$ and $\msOOO_1\neq\emptyset$.
\item
\mylabel{item:msooo dense}
$\msOOO_1\cup\msOOO_2$ is a dense subset of $\RRR^{u\times n\times m}$.
\item
\mylabel{item:closure o1}
$\overline\msOOO_1=\msCCC_t^{u\times n\times m}
=\RRR^{u\times n\times m}\setminus\msAAA_t^{u\times n\times m}$.
\item
\mylabel{item:mxy prime}
If $Y\in\msOOO_1\cup\msOOO_2$, then $I_t(M(\xxx,Y))\RRR[\xxx]$ is a prime ideal
of height $v$.
\item
\mylabel{item:real closed}
If $Y\in\msOOO_1$, then
$
\III(\VVV(I_t(M(\xxx,Y))))=I_t(M(\xxx,Y))
$.
\item
\mylabel{item:iv irr}
If $Y\in\msOOO_2$, then 
$
\III(\VVV(I_t(M(\xxx,Y))))=(x_1,\ldots,x_m)
$.
\end{enumerate}
\end{thm}
\begin{proof}
The set $\msAAA_t^{u\times n\times m}$ is an open subset of $\RRR^{u\times n\times m}$ 
by Corollary~\ref{cor:msa open} and $\tilde\msPPP_t$ is a {nonempty} open subset 
of $\RRR^{u\times n\times m}$ with 
$\msAAA_t^{u\times n\times m}\cap\tilde\msPPP_t=\emptyset$ by 
{Lemmas~\ref{lem:o1 dense} and \ref{lem:tilde p open}}.
Further $\msQQQ$ is a Zariski open subset.
Thus \ref{item:msooo open} holds.

\ref{item:msooo dense}:
Since $\msQQQ$ and $\msAAA_t^{{u\times n\times m}}\cup\tilde\msPPP_t$ are 
dense open subsets
of $\RRR^{u\times n\times m}$ by Corollary~\ref{cor:msa open} and 
Lemma \ref{lem:tilde p open}, we see that
$\msOOO_1\cup\msOOO_2\supset\msQQQ\cap(\msAAA_t^{{u\times n\times m}}\cup\msPPP_t)$ is also a
dense open subset of $\RRR^{u\times n\times m}$.

\ref{item:closure o1}:
Since  $\msQQQ$ is a 
dense subset of $\RRR^{u\times n\times m}$,
and $\tilde\msPPP_t$ is an open set, we see that
$\overline\msOOO_1=
\overline{\msQQQ\cap\tilde\msPPP_t}=\overline{\tilde\msPPP_t}=
\msCCC_t^{u\times n\times m}$
by Lemma \ref{lem:tilde p open}.

\ref{item:mxy prime} follows from Lemma \ref{lem:itmy prime}.

\ref{item:real closed}:
Assume the contrary and take $g\in\III(\VVV(I_t(M(\xxx,Y))))$ with 
$g\not\in I_t(M(\xxx,Y))$.
Set
$J=(g)\RRR[\xxx]+I_t(M(\xxx,Y))$.
Then $J\supsetneq I_t(M(\xxx,Y))\RRR[\xxx]$.
Since $I_t(M(\xxx,Y))\RRR[\xxx]$ is a prime ideal of height $v$
by Lemma \ref{lem:itmy prime}, 
we see that
$\height J>v$
and therefore
$
\RRR[x_1,\ldots, x_{m-v}]\cap J\neq (0)
$.

Take $0\neq f\in J\cap \RRR[x_1, \ldots, x_{m-v}]$.
Since $Y\in \tilde\msPPP_t$,
we can take $P\in\glin(u,\RRR)$ such that $PY\in\msPPP_t$.
Take $\bbb\in S_t(PY)$.
Since $O(\bbb,PY)$ defined in Lemma~\ref{lem:imp func}
 is an open set and $f$ is a non-zero polynomial,
we can take 
$(a_1, \ldots, a_{m-v})\in O(\bbb,PY)$ with $f(a_1, \ldots, a_{m-v})\neq 0$.
On the other hand, we see 
that there are
$a_{m-v+1}$, \ldots, $a_m\in\RRR$ such that 
$I_t(M(\aaa,Y))=I_t(PM(\aaa,Y))=I_t(M(\aaa,PY))=(0)$
by Lemma~\ref{lem:imp func},
 where 
$\aaa=(a_1,\ldots,a_m)$.
Thus by assumption, we see that $g(\aaa)=0$.
This contradicts to the fact that
$f\in J=(g)\RRR[\xxx]+I_t(M(\xxx,Y))\RRR[\xxx]$
and $f(a_1, \ldots, a_{m-v})\neq 0$.

Finally, \ref{item:iv irr} is clear from the definition of $\msAAA_t^{{u\times n\times m}}$.
\end{proof}


\section{Monomial preorder}
\mylabel{sec:mpo}

In this section, we introduce the notion of monomial preorder and prove a result
about ideals of minors by using it.

First we recall the notion of preorder.

\begin{definition}\rm
Let $S$ be a nonempty set and $\preceq$ a binary relation on $S$.
We say that $\preceq$ is a {{\em preorder}} on $S$ or $(S,\preceq)$ is a preordered
set if the following two conditions are satisfied.
\begin{enumerate}
\item
$a\preceq a$ for any $a\in S$.
\item
$a\preceq b$, $b\preceq c\Rightarrow a\preceq c$.
\setcounter{enumtemp}{\value{enumi}}\relax
\end{enumerate}
If moreover, 
\begin{enumerate}
\setcounter{enumi}{\value{enumtemp}}\relax
\item
$a\preceq b$ or $b\preceq a$ for any $a$, $b\in S$.
\end{enumerate}
is satisfied, then we say that $(S,\preceq)$ is a {{\em totally preordered set}}
or $\preceq$ is a {{\em total preorder}}.
\end{definition}

\begin{notation}
Let $(S,\preceq)$ be a preordered set.
We denote by $b\succeq a$ the fact $a\preceq b$.
We {denote} by $a\prec b$ or by $b\succ a$ the fact that $a\preceq b$ and $b\npreceq a$.
We also denote by $a\sim b$ the fact that $a\preceq b$ and $b\preceq a$.
\end{notation}

\begin{remark}\rm
The binary relation $\sim$ defined above is an equivalence relation and if
$a\sim a'$ and $b\sim b'$, then
$$
a\preceq b\iff a'\preceq b'.
$$
In particular, we can define a binary relation $\leq$ on the quotient set
$P=S/\sim$ by 
$\bar a\leq \bar b\stackrel{\mathrm{def}}{\iff}a\preceq b$, where $\bar a$ is the 
equivalence class which $a$ belongs to.
It is easily verified that $(P,\leq)$ is a partially ordered set and
$(S,\preceq)$ is a totally preordered set if and only if $(P,\leq)$ is a 
totally ordered set.
As usual, we denote $\bar a>\bar b$ the fact 
$\bar a\geq \bar b$ and $\bar a\neq \bar b$.
\end{remark}

\begin{definition}\rm
\mylabel{def:monom preorder}
Let $x_1$,  \ldots, $x_r$ be indeterminates.
We denote the set of monomials or power products of $x_1$, \ldots, $x_r$
by $\PP(x_1,\ldots, x_r)$.
A {{\em monomial preorder}} on $x_1$, \ldots, $x_r$ is a total preorder $\preceq$ on
$\PP(x_1,\ldots, x_r)$ satisfying the following conditions.
\begin{enumerate}
\item
$1\preceq m$ for any $m\in \PP(x_1, \ldots, x_r)$.
\item
\mylabel{item:second cond}
For $m_1$, $m_2$, $m\in \PP(x_1, \ldots, x_r)$,
$$
m_1\preceq m_2\iff m_1m\preceq m_2m.
$$
\end{enumerate}
Let $\sim$ be the equivalence relation on $\PP(x_1, \ldots, x_r)$ defined by the
monomial preorder $\preceq$.
We denote by $P(x_1, \ldots, x_r)$ the quotient set $\PP(x_1, \ldots, x_r)/\sim$
and by $\qdeg m$ the class of $m$ in $P(x_1,\ldots, x_r)$ and call it the {{\em quasi-degree}} of $m$,
where $m\in \PP(x_1,\ldots, x_r)$.
\end{definition}

\begin{remark}\rm
Our definition of monomial preorder may seem to be different from that of
\cite{Kemper:2014}, but it is identical except we allow $m\sim 1$ for a
monomial $m \neq 1$.
\end{remark}

\begin{example}[c.f.\ {\cite[Example 3.1]{Kemper:2014}}]\rm
\mylabel{ex:kt3.1}
Let $x_1$, \ldots, $x_r$ be indeterminates, $W=(\www_1, \ldots, \www_s)$
an $r\times s$ matrix whose entries are real numbers such that the first
nonzero entry of each row is positive.
If one defines
$$
x^\aaa\preceq x^\bbb\stackrel{\mathrm{def}}{\iff}
(\aaa\cdot\www_1,\ldots, \aaa\cdot\www_s)\lexeq
(\bbb\cdot\www_1,\ldots, \bbb\cdot\www_s),
$$
where $\cdot$ denotes the inner product and $\lexeq$ denotes the lexicographic 
order, 
then $\preceq$ is a monomial preorder.
In fact, one can prove by the same way as \cite{Robbiano:1985}
that every monomial preorder is of this type.
\end{example}

\begin{definition}\rm
Let $K$ be a {field} and $x_1$, \ldots, $x_r$ indeterminates.
If a monomial preorder on $x_1$, \ldots, $x_r$ is defined, we say that 
$K[x_1, \ldots, x_r]$ is a polynomial ring with monomial preorder.
Let $f$ be a nonzero element of $K[x_1, \ldots, x_r]$.
We say that $f$ is a {{\em form}} if all the monomials appearing in $f$ have the
same quasi-degree.
We denote by $\qdeg f$ the quasi-degree of the monomials appearing in $f$.
Let $g$ be a nonzero element of $K[x_1,\ldots, x_r]$.
Then there is a unique expression
$$
g=g_1+g_2+\cdots+g_t
$$
of $g$, where $g_i$ is a form for $1\leq i\leq t$ and
$\qdeg g_1>\qdeg g_2>\cdots>\qdeg g_t$.
We define the leading form of $g$, denoted $\lf(g)$ as $g_1$.
\end{definition}

\begin{remark}\rm
Let $K[x_1,\ldots, x_r]$ be a polynomial ring with monomial preorder 
and $f$, $g$ nonzero elements of $K[x_1, \ldots, x_r]$.
Then $\lf(fg)=\lf(f)\lf(g)$.
\end{remark}

\begin{remark}\rm
It is essential to assume both implications in \ref{item:second cond} of
Definition \ref{def:monom preorder}.
For example, let $x$ and $y$ be indeterminates.
We define total preorder on $\PP(x,y)$ by
$1\prec y\prec x$ and
$m_1\prec m_2$ if the total degree of $m_1$ is less than that of $m_2$.
Then it is easily verified that
\begin{enumerate}
\item
$1\prec m$ for any $m\in \PP(x,y)\setminus\{1\}$.
\item
$m_1\preceq m_2\Rightarrow m_1m\preceq m_2m$.
\end{enumerate}
Let $f=x+y$.
Then $\lf(f)=x$ while $\lf(f^2)=x^2+2xy+y^2\neq x^2=(\lf(f))^2$.
\end{remark}

\begin{definition}\rm
Let $x_1$, \ldots, $x_r$ be indeterminates.
Suppose that a total preorder on $\{x_1, \ldots, x_r\}$
is defined.
Rewrite the set $\{x_1, \ldots, x_r\}$ as follows.
$\{x_1, \ldots, x_r\}=\{y_{11}, \ldots, y_{1s_1}, y_{21}, \ldots, y_{2s_s},
\ldots, y_{t1}, \ldots, y_{ts_t}\}$,
$s_1+\cdots+s_t=r$,
$y_{11}\sim\cdots\sim y_{1s_1}\succ y_{21}\sim\cdots\sim y_{2s_2}\succ\cdots\succ
y_{t1}\sim\cdots\sim y_{ts_t}$.

The lexicographic monomial preorder on $\PP(x_1,\ldots, x_r)$ is defined as follows.
$\prod_{i=1}^t\prod_{j=1}^{s_i}y_{ij}^{a_{ij}}\preceq
\prod_{i=1}^t\prod_{j=1}^{s_i}y_{ij}^{b_{ij}}$ if and only if one of the following
conditions is satisfied.
\begin{itemize}
\item
$\sum_{j=1}^{s_1}a_{1j}<\sum_{j=1}^{s_1}b_{1j}$.
\item
$\sum_{j=1}^{s_1}a_{1j}=\sum_{j=1}^{s_1}b_{1j}$ and
$\sum_{j=1}^{s_2}a_{2j}<\sum_{j=1}^{s_2}b_{2j}$.
\item
$\sum_{j=1}^{s_1}a_{1j}=\sum_{j=1}^{s_1}b_{1j}$,
$\sum_{j=1}^{s_2}a_{2j}=\sum_{j=1}^{s_2}b_{2j}$ and
$\sum_{j=1}^{s_3}a_{3j}<\sum_{j=1}^{s_3}b_{3j}$.
\\
\quad $\vdots$
\item
$\sum_{j=1}^{s_1}a_{1j}=\sum_{j=1}^{s_1}b_{1j}$,
$\sum_{j=1}^{s_2}a_{2j}=\sum_{j=1}^{s_2}b_{2j}$,
$\ldots$,
$\sum_{j=1}^{s_{t-2}}a_{t-2,j}=\sum_{j=1}^{s_{t-2}}b_{t-2,j}$ and
$\sum_{j=1}^{s_{t-1}}a_{t-1,j}<\sum_{j=1}^{s_{t-1}}b_{t-1,j}$.
\item
$\sum_{j=1}^{s_1}a_{1j}=\sum_{j=1}^{s_1}b_{1j}$,
$\sum_{j=1}^{s_2}a_{2j}=\sum_{j=1}^{s_2}b_{2j}$,
$\ldots$,
$\sum_{j=1}^{s_{t-1}}a_{t-1,j}=\sum_{j=1}^{s_{t-1}}b_{t-1,j}$ and
$\sum_{j=1}^{s_{t}}a_{t,j}\leq\sum_{j=1}^{s_{t}}b_{t,j}$.
\end{itemize}
\end{definition}

\begin{remark}\rm
Suppose that $x_1$, \ldots, $x_r$ are indeterminates and total preorder
$\preceq$ on $\{x_1, \ldots, x_r\}$ is defined.
Suppose also that
$$
x_1\sim\cdots\sim x_{m_1}\succ x_{m_1+1}\sim\cdots\sim x_{m_2}\succ\cdots
\succ x_{m_{t-1}+1}\sim\cdots\sim x_{m_t},
$$
$m_t=r$.
Then the lexicographic monomial preorder induced by this preorder on
$\{x_1, \ldots, x_r\}$ is the one defined as in Example \ref{ex:kt3.1}
by the $r\times t$ matrix whose
$j$-th column has 1 in $m_{j-1}+1$ through $m_{j}$-th position and 0
in others, where we set $m_0=0$.
\end{remark}

\begin{definition}\rm
We set
$$
[a_1, \ldots, \overset{k}{a_i}, \ldots, a_n]\define
[a_1, \ldots, a_{i-1}, k, a_{i+1}, \ldots, a_n]
$$
and
$$
[a_1,\ldots, \overset{k}{a_i}, \ldots, \overset{l}{a_j}, \ldots, a_n]
\define
[a_1,\ldots, a_{i-1}, k, a_{i+1}, \ldots, a_{j-1}, l, a_{j+1}, \ldots, a_n].
$$
\end{definition}

\begin{lemma}
\mylabel{lem:lib}
Let $K$ be a field, $K[x_1, \ldots, x_r]$ a polynomial ring with monomial preorder,
$S$ a subset of $\{x_1, \ldots, x_r\}$ and $g_1$,  \ldots, $g_t\in K[x_1, \ldots, x_r]$.
Set $L=K[S]$.
If $\lf(g_1)$, \ldots, $\lf(g_t)$ are linearly independent over $L$,
then $g_1$, \ldots, $g_t$ are linearly independent over $L$.
\end{lemma}
\begin{proof}
Assume the contrary and suppose that
$$
\sum_i c_i g_i=0
$$
is a non-trivial relation
where $c_i\in L$ 
and $c_i\neq 0$ for any $i$ which appears in the above sum.
Then
$$
\sum\nolimits '\lf(c_i)\lf(g_i)=0,
$$
where $\sum'$ runs through $i$'s with $\qdeg\lf(c_ig_i)$ are maximal.
Since $\lf(c_i)\in L$ for any $i$, it contradicts the assumption.
\end{proof}

\begin{lemma}[{Pl\"ucker} relations, see e.g.\ {\cite[(4,4) Lemma]{Bruns-Vetter:1988}}]
\mylabel{lem:Pluecker relation}
For every $u \times n$-matrix M, $u \geq n$, with entries
in a commutative ring and all indices $a_1,\ldots,a_k,b_l,\ldots,b_n,c_1,\ldots,c_s \in 
\{1,\ldots,u\}$ such that $s = n - k + l - 1 > n$, $t = n - k > 0$ 
one has
$$
\sum_{\genfrac{}{}{0pt}{2}{i_1<\cdots<i_t}{\genfrac{}{}{0pt}{2}{i_{t+1}<\cdots<i_s}{\{1,\ldots,s\}=\{i_1,\ldots,i_s\}}}}
{\mathrm{sgn}}(i_1,\ldots,i_s)[a_1,\ldots,a_k,c_{i_1},\ldots,c_{i_t}]_M[c_{i_{t+1}},\ldots,c_{i_s},b_l,\ldots,b_n]_M=0,$$
where $\mathrm{sgn}(\sigma)$ is the signature of a permutation $\sigma$
and the notations are defined in Definition~\ref{def:unmtv}.
\end{lemma}

An element $a$ is called a {\em non-zerodivisor} if
$ab=0$ implies $b=0$.

\begin{lemma}
\mylabel{lem:lc}
Let $A=A_0\oplus A_1\oplus\cdots$ be a graded ring, $X=(x_{ij})$ a $u\times n$ matrix
with $u>n$ and entries in $A_1$ and 
$\yyy=(y_1, \ldots, y_n)\in A_1^{1\times n}$.
Set 
$\tilde X=\begin{pmatrix}X\\\yyy\end{pmatrix}$
and $\Gamma=\Gamma(u\times n)$.
Suppose that
\begin{itemize}
\item
$\delta=[1,2,\ldots, n]_X$ is a non-zerodivisor of $A$,
\item
for any $k_1$ and $k_2$ with $1\leq k_1<k_2\leq n$,
\begin{eqnarray*}
&&\delta_X\gamma_X\quad\mbox{\rm ($\gamma\in\Gamma$),}\\
&&[1, \ldots, \overset{n+1}{k_2}, \ldots, n]_X\gamma_X\quad
\mbox{\rm ($\gamma\in\Gamma\setminus\{\delta\}$) and}\\
&&[1, \ldots, \overset{n+1}{k_1}, \ldots, n]_X\gamma_X\quad\mbox{\rm ($\gamma\in\Gamma$, 
$\supp\gamma\not\supset\{1, \ldots, \hat k_2, \ldots, n\}$)}
\end{eqnarray*}
are linearly independent over $A_0$ and
\item
$I_n(\tilde X)=I_n(X)$,
\end{itemize}
{where the notations are defined in Definition~\ref{def:unmtv}.}
Then, $\yyy$ is an $A_0$-linear combination of rows of $X$.
\end{lemma}
\begin{proof}
We denote $\gamma_{\tilde X}$ as $\gamma$ 
and $\displaystyle \sum_{\gamma \in\Gamma}$ as $\displaystyle\sum_\gamma$
for simplicity.
Set
$$
[1,\ldots, \overset{u+1}k, \ldots, n]=\sum_\gamma a_\gamma^{(k)}\gamma
$$
and
$$
[1, \ldots, \overset{n+1}{k_1}, \ldots, \overset{u+1}{k_2}, \ldots, n]
=\sum_\gamma a_\gamma^{(k_1,k_2)}\gamma
$$
where $a_\gamma^{(k)}$, $a_\gamma^{(k_1,k_2)}\in A$.
By considering the degree, we may assume that $a_\gamma^{(k)}$, $a_\gamma^{(k_1,k_2)}\in A_0$.

\begin{eqnarray*}
&&[1,\ldots, \overset{n+1}{k_1}, \ldots, \overset{u+1}{k_2}, \ldots, n]_{\tilde X\cof(X^{\leq n})}
\\
&=&\det(\cof(X^{\leq n}))[1, \ldots, \overset{n+1}{k_1}, \ldots, \overset {u+1}{k_2}, \ldots, n]
\\
&=&\delta^{n-1}\sum_\gamma a_\gamma^{(k_1,k_2)}\gamma,
\end{eqnarray*}
{where $\cof(X^{\leq n})$ denotes the matrix of cofactors of $X^{\leq n}$.} 
On the other hand, since
$$
\tilde X\cof(X^{\leq n})=
\begin{pmatrix}
\delta\\
&\delta\\
&&\ddots\\
&&&\delta\\
[\overset{n+1}1, 2\ldots, n]&[1, \overset{n+1}2, \ldots, n]&\cdots&[1,2,\ldots,\overset{n+1}n]\\
[\overset{n+2}1, 2\ldots, n]&[1, \overset{n+2}2, \ldots, n]&\cdots&[1,2,\ldots,\overset{n+2}n]\\
\vdots&\vdots&\vdots&\vdots\\
[\overset{u+1}1, 2\ldots, n]&[1, \overset{u+1}2, \ldots, n]&\cdots&[1,2,\ldots,\overset{u+1}n]
\end{pmatrix},
$$
we see that
\begin{eqnarray*}
&&[1,\ldots, \overset{n+1}{k_1}, \ldots, \overset{u+1}{k_2}, \ldots, n]_{\tilde X\cof(X^{\leq n})}
\\
&=&\delta^{n-2}
\det
\begin{pmatrix}
[1,\ldots, \overset{n+1}{k_1},\ldots, n]&[1,\ldots, \overset{n+1}{k_2},\ldots, n]\\
[1,\ldots, \overset{u+1}{k_1},\ldots, n]&[1,\ldots, \overset{u+1}{k_2},\ldots, n]
\end{pmatrix}.
\end{eqnarray*}

Since $\delta$ is a non-zerodivisor, we see that 
\begin{eqnarray*}
&&
\sum_\gamma a_\gamma^{(k_1,k_2)}\delta\gamma\\
&=&
[1, \ldots,\overset{n+1}{k_1}, \ldots, n][1,\ldots, \overset{u+1}{k_2}, \ldots, n]
-[1, \ldots,\overset{n+1}{k_2}, \ldots, n][1,\ldots, \overset{u+1}{k_1}, \ldots, n]
\\
&=&\sum_\gamma a_\gamma^{(k_2)}[1,\ldots, \overset{n+1}{k_1},\ldots,n]\gamma
-\sum_\gamma a_\gamma^{(k_1)}[1,\ldots, \overset{n+1}{k_2},\ldots,n]\gamma
\\
&=&
-a_\delta^{(k_1)}\delta[1,\ldots,\overset{n+1}{k_2},\ldots,n]
-\sum_{\gamma\in\Gamma\setminus\{\delta\}} 
a_\gamma^{(k_1)}[1,\ldots, \overset{n+1}{k_2},\ldots,n]\gamma
\\
&&
+a_\delta^{(k_2)}\delta[1,\ldots,\overset{n+1}{k_1},\ldots,n]
+\sum_{\supp\gamma\cap\{1,\ldots,n\}=\{1,\ldots,\hat k_2,\ldots, n\}}
 a_\gamma^{(k_2)}[1,\ldots, \overset{n+1}{k_1},\ldots,n]\gamma
\\
&&+\sum_{\supp\gamma\not\supset\{1,\ldots,\hat k_2, \ldots,n\}}
 a_\gamma^{(k_2)}[1,\ldots, \overset{n+1}{k_1},\ldots,n]\gamma
\end{eqnarray*}

Suppose that $\supp \gamma\cap\{1, \ldots, n\}=\{1,\ldots,\hat k_2, \ldots, n\}$.
Then $\gamma=[1, \ldots, \hat k_2, \ldots, n,l]$ for some $l$ with
$n+1\leq l \leq u$.
Thus $\gamma=(-1)^{n-k_2}[1,\ldots,\overset l {k_2}, \ldots, n]$.
By applying Lemma~\ref{lem:Pluecker relation}
to
$$
[1,\ldots,\overset{l}{k_2},\ldots,n][1,\ldots,\overset{n+1}{k_1},\ldots,n],
$$
by substituting $u$, $n$, $k$, $s$, $l$ 
by $u$, $n$, $k_1-1$, $n+1$, $k_1+1$ respectively,
we see that 
\begin{eqnarray*}
&&[1,\ldots,\overset{l}{k_2},\ldots,n][1, \ldots,\overset{n+1}{k_1},\ldots,n]\\
&=&\delta[1,\ldots,\overset{n+1}{k_1}\ldots,\overset{l}{k_2},\ldots,n]
+[1,\ldots,\overset{n+1}{k_2},\ldots,n][1,\ldots,\overset{l}{k_1},\ldots,n]\\
&=&\delta[1,\ldots,\overset{n+1}{k_1}\ldots,\overset{l}{k_2},\ldots,n]
+(-1)^{n-k_1}[1,\ldots,\overset{n+1}{k_2},\ldots,n][1,\ldots,\hat k_1,\ldots,n,l]
\end{eqnarray*}
Therefore, if we set
$$
b_\gamma^{(k_1)}=
\left\{
\begin{array}{ll}
a_\gamma^{(k_1)}
-(-1)^{k_2-k_1}a_{[1,\ldots,\hat k_2,\ldots,n,l]}^{(k_2)}
\quad&\mbox{if $\gamma=[1,\ldots,\hat k_1\ldots,n,l]$}\\
a_\gamma^{(k_1)}\quad&\mbox{otherwise}
\end{array}
\right.
$$
for $\gamma\in\Gamma\setminus\{\delta\}$,
we see that there are $b_\gamma\in A_0$ such that
\begin{eqnarray*}
&&\sum_\gamma b_\gamma \delta\gamma
-\sum_{\gamma\in\Gamma\setminus\{\delta\}}b_\gamma^{(k_1)}[1,\ldots,\overset{n+1}{k_2},\ldots,n]\gamma\\
&&\quad
+\sum_{\supp\gamma\not\supset\{1,\ldots,\hat k_2,\ldots,n\}}a_\gamma^{(k_2)}[1,\ldots,\overset{n+1}{k_1},\ldots,n]\gamma
=0.
\end{eqnarray*}

Thus we see, by the assumption, that
$$
b_\gamma^{(k_1)}=0\quad\mbox{if $\gamma\in\Gamma\setminus\{\delta\}$}
$$
and
$$
a_\gamma^{(k_2)}=0\quad\mbox{if $\supp\gamma\not\supset\{1,\ldots,\hat k_2, \ldots, n\}$}.
$$
Therefore, by the definition of $b_\gamma^{(k_1)}$, we see that 
$$
(-1)^{n-k_1}a_{[1,\ldots,\hat k_1,\ldots, n,l]}^{(k_1)}=
(-1)^{n-k_2}a_{[1,\ldots,\hat k_2,\ldots, n,l]}^{(k_2)}
$$
and
$$
a_\gamma^{(k_1)}=0\quad\mbox{if $\supp\gamma\not\supset\{1,\ldots,\hat k_1, \ldots, n\}$}.
$$

Since these hold for any $k_1$, $k_2$ with $1\leq k_1<k_2\leq n$, we see that if we set
$$
c_l=a_{[1,\ldots, n-1,l]}^{(n)}
$$
for $l$ with $n+1\leq l\leq u$,
\begin{eqnarray*}
[1,\ldots,\overset{u+1} k, \ldots, n]&=&
a_\delta^{(k)}{\delta}+(-1)^{n-k}\sum_{l=n+1}^u c_l[1,\ldots,\hat k,\ldots,n,l]
\\
&=&a_\delta^{(k)}\delta+\sum_{l=n+1}^u c_l[1, \ldots, \overset l k,\ldots, n]
\end{eqnarray*}
for any $k$ with $1\leq k\leq n$.

Set 
$$
\zzz=\yyy-\sum_{s=1}^n a_\delta^{(s)}X^{=s}-\sum_{l=n+1}^u c_l X^{=l}
$$
and
$$
Z=\begin{pmatrix}X\\\zzz\end{pmatrix}.
$$
Then
\begin{eqnarray*}
&&[1,\ldots,\overset{u+1}k, \ldots,n]_Z\\
&=&[1, \ldots, \overset{u+1}k, \ldots, n]_{\tilde X}
-\sum_{s=1}^n a_\delta^{(s)}[1,\ldots \overset sk, \ldots n]_X
-\sum_{l=n+1}^u c_l[1,\ldots \overset lk, \ldots n]_X
\\
&=&0
\end{eqnarray*}
for any $k$ with $1\leq k\leq n$.
Thus $\zzz\cof(X^{\leq n})=\zerovec$ and we see that $\zzz=\zerovec$ since
$\det\cof(X^{\leq n})=\delta^{n-1}$ is a non-zerodivisor.
\end{proof}

\begin{lemma}
\mylabel{lem:li}
Let $L$ be a field, $n$, $m$ integers with $n\geq m\geq 3$,
$x_1$, \ldots, $x_m$ indeterminates and $\alpha$, $\beta\in L$ with $0\neq \alpha\neq\beta\neq 0$.
Then the following polynomials are linearly independent over $L$.
\begin{eqnarray*}
&&x_1^{2n}\\
&&x_1^{2n-1}(x_2-\beta x_m)\\
&&x_1^{2n-1}x_s\quad\mbox{\rm ($3\leq s\leq m-1$)}\\
&&x_1^n(x_2-\alpha x_m)(x_2-\beta x_m)x_{b_1}\cdots x_{b_{n-2}}\quad
\mbox{\rm ($1\leq b_1\leq \cdots\leq b_{n-2}\leq 2$)}\\
&&x_1^n(x_2-\alpha x_m)x_{b_1}\cdots x_{b_{n-2}}x_s\quad
\mbox{\rm ($1\leq b_1\leq \cdots\leq b_{n-2}\leq 2$, $3\leq s \leq m-1$)}\\
&&x_1^nx_{b_1}\cdots x_{b_{n}}\quad
\mbox{\rm ($1\leq b_1\leq \cdots\leq b_{n}\leq m-1$, $b_{n-1}\geq 3$)}\\
&&x_1^{2n-2}(x_2-\beta x_m)^2\\
&&x_1^{2n-2}(x_2-\beta x_m)x_s\quad\mbox{\rm ($3\leq s\leq m-1$)}\\
&&x_1^{n-1}(x_2-\alpha x_m)(x_2-\beta x_m)^2x_{b_1}\cdots x_{b_{n-2}}\quad
\mbox{\rm ($1\leq b_1\leq \cdots\leq b_{n-2}\leq 2$)}\\
&&x_1^{n-1}(x_2-\alpha x_m)(x_2-\beta x_m)x_{b_1}\cdots x_{b_{n-2}}x_s\quad
\mbox{\rm ($1\leq b_1\leq \cdots\leq b_{n-2}\leq 2$, $3\leq s\leq m-1$)}\\
&&x_1^{n-1}(x_2-\beta x_m)x_{b_1}\cdots x_{b_{n}}\quad
\mbox{\rm ($1\leq b_1\leq \cdots\leq b_{n}\leq m-1$, $b_{n-1}\geq 3$)}\\
&&x_1^{n-2}(x_2-\alpha x_m)^2(x_2-\beta x_m)^2x_{b_1}\cdots x_{b_{n-2}}\quad
\mbox{\rm ($1\leq b_1\leq \cdots\leq b_{n-2}\leq 2$)}\\
&&x_1^{n-2}(x_2-\alpha x_m)^2(x_2-\beta x_m)x_{b_1}\cdots x_{b_{n-2}}x_s\quad
\mbox{\rm ($1\leq b_1\leq \cdots\leq b_{n-2}\leq 2$, $3\leq s\leq m-1$)}\\
&&x_1^{n-2}(x_2-\alpha x_m)(x_2-\beta x_m)x_{b_1}\cdots x_{b_{n}}\quad
\mbox{\rm ($1\leq b_1\leq \cdots\leq b_{n}\leq m-1$, $b_{n-1}\geq 3$)}
\end{eqnarray*}
\end{lemma}
\begin{proof}
Set $\deg x_1=\deg x_2=\deg x_m=0$ and
$\deg x_3=\cdots=\deg x_{m-1}=1$.
Then the polynomials under consideration are homogeneous.
Thus it is enough to show that for each integer $d$, the polynomials of degree $d$
in the above list are linearly independent over $L$.

First consider the polynomials with degree more than 1.
They are
\begin{eqnarray*}
&&x_1^nx_{b_1}\cdots x_{b_{n}}\quad
\mbox{($1\leq b_1\leq \cdots\leq b_{n}\leq m-1$, $b_{n-1}\geq 3$)}\\
&&x_1^{n-1}(x_2-\beta x_m)x_{b_1}\cdots x_{b_{n}}\quad
\mbox{($1\leq b_1\leq \cdots\leq b_{n}\leq m-1$, $b_{n-1}\geq 3$)}\\
&&x_1^{n-2}(x_2-\alpha x_m)(x_2-\beta x_m)x_{b_1}\cdots x_{b_{n}}\quad
\mbox{($1\leq b_1\leq \cdots\leq b_{n}\leq m-1$, $b_{n-1}\geq 3$)}.
\end{eqnarray*}
By first substituting $\beta^{-1}x_2$ to $x_m$ and next by substituting $\alpha^{-1}x_2$
to $x_m$ one sees that these polynomials are linearly independent.

Next consider the polynomials with degree 1.
They are
\begin{eqnarray*}
&&x_1^{2n-1}x_s\quad\mbox{($3\leq s\leq m-1$)}\\
&&x_1^n(x_2-\alpha x_m)x_{b_1}\cdots x_{b_{n-2}}x_s\quad
\mbox{($1\leq b_1\leq \cdots\leq b_{n-2}\leq 2$, $3\leq s \leq m-1$)}\\
&&x_1^{2n-2}(x_2-\beta x_m)x_s\quad\mbox{($3\leq s\leq m-1$)}\\
&&x_1^{n-1}(x_2-\alpha x_m)(x_2-\beta x_m)x_{b_1}\cdots x_{b_{n-2}}x_s\quad
\\
&&\hskip20mm
\mbox{($1\leq b_1\leq \cdots\leq b_{n-2}\leq 2$, $3\leq s\leq m-1$)}\\
&&x_1^{n-2}(x_2-\alpha x_m)^2(x_2-\beta x_m)x_{b_1}\cdots x_{b_{n-2}}x_s\quad
\\
&&\hskip20mm
\mbox{($1\leq b_1\leq \cdots\leq b_{n-2}\leq 2$, $3\leq s\leq m-1$)}.
\end{eqnarray*}
By a similar but more subtle argument as above, 
one sees that these polynomials are linearly independent.

Finally consider the polynomials with degree 0.
They are
\begin{eqnarray*}
&&x_1^{2n}\\
&&x_1^{2n-1}(x_2-\beta x_m)\\
&&x_1^n(x_2-\alpha x_m)(x_2-\beta x_m)x_{b_1}\cdots x_{b_{n-2}}\quad
\mbox{($1\leq b_1\leq \cdots\leq b_{n-2}\leq 2$)}\\
&&x_1^{2n-2}(x_2-\beta x_m)^2\\
&&x_1^{n-1}(x_2-\alpha x_m)(x_2-\beta x_m)^2x_{b_1}\cdots x_{b_{n-2}}\quad
\mbox{($1\leq b_1\leq \cdots\leq b_{n-2}\leq 2$)}\\
&&x_1^{n-2}(x_2-\alpha x_m)^2(x_2-\beta x_m)^2x_{b_1}\cdots x_{b_{n-2}}\quad
\mbox{($1\leq b_1\leq \cdots\leq b_{n-2}\leq 2$)}\\
\end{eqnarray*}
By a similar but more subtle argument,
one sees that these polynomials are linearly independent.
\end{proof}

\begin{lemma}
\mylabel{lem:lct}
Let $K$ be a field, $T=(T_1;\cdots;T_m)=(t_{ijk})$ a $u\times n\times m$-tensor of
indeterminates with $u>n\geq m\geq u-n+2$, $\xxx=(x_1$, \ldots, $x_m)$
a vector of indeterminates.
Set
$
X=M(\xxx,T)
$ (c.f.\ Definition \ref{def:unmtv}),
$\Gamma=\Gamma(u\times n)$,
$\delta=\delta_0=[1,\ldots,n]$ and
$A=K[T][\xxx]$.
Then $\delta_X$ is a non-zerodivisor of $A$ and for any $k_1$, $k_2$ with
$1\leq k_1, k_2\leq n$, $k_1\neq k_2$, 
\begin{eqnarray*}
&&\delta_X\gamma_X\quad\mbox{($\gamma\in\Gamma$)}\\
&&[1,\ldots,\overset{n+1} {k_2}, \ldots, n]_X\gamma_X\quad\mbox{\rm ($\gamma\in\Gamma\setminus\{\delta\}$)}\\
&&[1,\ldots,\overset{n+1} {k_1}, \ldots, n]_X\gamma_X\quad\mbox{\rm ($\gamma\in\Gamma$,
$\supp\gamma\not\supset\{1,\ldots,\hat k_2, \ldots,n\}$)}
\end{eqnarray*}
are linearly independent over $K[T]$.
\end{lemma}

\begin{proof}
The first assertion is clear since $A$ is an integral domain and $\delta_X\neq 0$.
Next we prove the second assertion.
By symmetry, we may assume that $k_1=n-1$ and $k_2=n$.
Set $\delta_1=[1,2,\ldots, n-1,n+1]$ and $\delta_2=[1,2,\ldots,n-2,n,n+1]$.
We introduce the lexicographic monomial preorder induced by the preorder on the indeterminates 
defined as follows.

%
%
If one of the following is satisfied, we define $t_{ijk}\succ t_{i'j'k'}$.
\begin{itemize}
\item
$i<i'$.
\item
$i=i'$ and $j<j'$.
\item
$i=i'$, $j=j'$, $k<k'$ and ``$i<j$ or $i>j+u-n$''.
\end{itemize}
In case $0\leq i-j\leq u-n$ and $(i,j)\neq (n,n-1), (n+1,n)$, we define
$$
t_{i,j,i-j+1}\succ t_{i,j,i-j+2}\succ \cdots\succ t_{i,j,m}\succ t_{i,j,1}\succ\cdots\succ t_{i,j,i-j}.
$$
In case $(i,j)=(n,n-1)$ or $(n+1,n)$, we define
$$
t_{i,j,2}\sim t_{i,j,m}\succ t_{i,j,1}\succ t_{i,j,3}\succ\cdots\succ t_{i,j,m-1}.
$$
And
$$
t_{ijk}\succ x_1\sim x_2\sim\cdots \sim x_m
$$
for any $i$, $j$ and $k$.

Set
\begin{eqnarray*}
&&\Gamma_0=\{[a_1,\ldots,a_n]\in \Gamma\mid a_{n-1}=n-1\},\\
&&\Gamma_1=\{[a_1,\ldots,a_n]\in \Gamma\mid a_{n-1}=n\},\\
&&\Gamma_2=\{[a_1,\ldots,a_n]\in \Gamma\mid a_{n-1}\geq n+1\},\\
&&\Gamma_{00}=\{[a_1,\ldots,a_n]\in \Gamma_0\mid a_{n}=n\}=\{\delta_0\},\\
&&\Gamma_{01}=\{[a_1,\ldots,a_n]\in \Gamma_0\mid a_{n}=n+1\}=\{\delta_1\},\\
&&\Gamma_{02}=\{[a_1,\ldots,a_n]\in \Gamma_0\mid a_{n}\geq n+2\},\\
&&\Gamma_{11}=\{[a_1,\ldots,a_n]\in \Gamma_1\mid a_{n}=n+1\}=\{\delta_2\},\\
&&\Gamma_{12}=\{[a_1,\ldots,a_n]\in \Gamma_1\mid a_{n}\geq n+2\}.
\end{eqnarray*}
Then
\begin{eqnarray*}
&&\Gamma=\Gamma_0\sqcup \Gamma_1\sqcup\Gamma_2\\
&&\Gamma_0=\Gamma_{00}\sqcup\Gamma_{01}\sqcup\Gamma_{02}\\
&&\Gamma_1=\Gamma_{11}\sqcup\Gamma_{12}.
\end{eqnarray*}
Set
$\alpha_1=t_{n,n-1,2}$, $\alpha_2=t_{n,n-1,m}$,
$\beta_1=t_{n+1,n,2}$ and $\beta_2=t_{n+1,n,m}$.
Then for $\gamma=[a_1,\ldots, a_n]\in\Gamma$, $\lf(\gamma_X)$ is,
up to multiplication of nonzero element of $K[T]$, as follows.
\begin{eqnarray*}
&&x_1^n\quad\mbox{if $\gamma\in\Gamma_{00}$}\\
&&x_1^{n-1}(\beta_1x_2+\beta_2x_m)\quad\mbox{if $\gamma\in\Gamma_{01}$}\\
&&x_1^{n-1}x_{a_n-n+1}\quad\mbox{if $\gamma\in\Gamma_{02}$}\\
&&x_{a_1}x_{a_2-1}\cdots x_{a_{n-2}-n+3}(\alpha_1x_2+\alpha_2 x_m)(\beta_1 {x_2+}\beta_2x_m)
\quad\mbox{if $\gamma\in\Gamma_{11}$}\\
&&x_{a_1}x_{a_2-1}\cdots x_{a_{n-2}-n+3}(\alpha_1x_2+\alpha_2 x_m)x_{a_n-n+1}
\quad\mbox{if $\gamma\in\Gamma_{12}$}\\
&&x_{a_1}x_{a_2-1}\cdots x_{a_{n}-n+1}
\quad\mbox{if $\gamma\in\Gamma_{2}$}.
\end{eqnarray*}
Therefore, for $\gamma=[a_1, \ldots, a_n]\in\Gamma$, $\lf(\delta_X\gamma_X)$ is,
up to multiplication of nonzero element of $K[T]$,
\begin{eqnarray*}
&&x_1^{2n}\quad\mbox{if $\gamma\in\Gamma_{00}$}\\
&&x_1^{2n-1}(\beta_1x_2+\beta_2x_m)\quad\mbox{if $\gamma\in\Gamma_{01}$}\\
&&x_1^{2n-1}x_{a_n-n+1}\quad\mbox{if $\gamma\in\Gamma_{02}$}\\
&&x_1^nx_{a_1}x_{a_2-1}\cdots x_{a_{n-2}-n+3}(\alpha_1x_2+\alpha_2 x_m)(\beta_1 {x_2+}\beta_2x_m)
\quad\mbox{if $\gamma\in\Gamma_{11}$}\\
&&x_1^nx_{a_1}x_{a_2-1}\cdots x_{a_{n-2}-n+3}(\alpha_1x_2+\alpha_2 x_m)x_{a_n-n+1}
\quad\mbox{if $\gamma\in\Gamma_{12}$}\\
&&x_1^nx_{a_1}x_{a_2-1}\cdots x_{a_{n}-n+1}
\quad\mbox{if $\gamma\in\Gamma_{2}$}.
\end{eqnarray*}
For $\gamma=[a_1, \ldots, a_n]\in\Gamma\setminus\{\delta\}$, $\lf((\delta_1)_X\gamma_X)$ is,
up to multiplication of nonzero element of $K[T]$,
\begin{eqnarray*}
&&x_1^{2n-2}(\beta_1x_2+\beta_2x_m)^2\quad\mbox{if $\gamma\in\Gamma_{01}$}\\
&&x_1^{2n-2}(\beta_1x_2+\beta_2x_m)x_{a_n-n+1}\quad\mbox{if $\gamma\in\Gamma_{02}$}\\
&&x_1^{n-1}x_{a_1}x_{a_2-1}\cdots x_{a_{n-2}-n+3}(\alpha_1x_2+\alpha_2 x_m)(\beta_1 {x_2+}\beta_2x_m)^2
\quad\mbox{if $\gamma\in\Gamma_{11}$}\\
&&x_1^{n-1}x_{a_1}x_{a_2-1}\cdots x_{a_{n-2}-n+3}(\alpha_1x_2+\alpha_2 x_m)(\beta_1 {x_2+}\beta_2x_m)x_{a_n-n+1}
\quad\mbox{if $\gamma\in\Gamma_{12}$}\\
&&x_1^{n-1}(\beta_1 {x_2+}\beta_2x_m)x_{a_1}x_{a_2-1}\cdots x_{a_{n}-n+1}
\quad\mbox{if $\gamma\in\Gamma_{2}$}.
\end{eqnarray*}
Finally, consider the leading form of $(\delta_2)_X\gamma_X$, where
$\gamma=[a_1, \ldots, a_n]\in\Gamma$ and $\supp\gamma\not\supset\{1, \ldots, n-1\}$.
It is easily verified that $\supp\gamma\not\supset\{1,\ldots,n-1\}$ if
and only if $\gamma\in\Gamma_1\sqcup\Gamma_2$.
Thus the leading form of $(\delta_2)_X\gamma_X$ is, up to multiplication of 
nonzero element of $K[T]$,
\begin{eqnarray*}
&&x_1^{n-2}x_{a_1}x_{a_2-1}\cdots x_{a_{n-2}-n+3}(\alpha_1x_2+\alpha_2 x_m)^2
(\beta_1 {x_2+}\beta_2x_m)^2
\quad\mbox{if $\gamma\in\Gamma_{11}$}\\
&&x_1^{n-2}x_{a_1}x_{a_2-1}\cdots x_{a_{n-2}-n+3}(\alpha_1x_2+\alpha_2 x_m)^2(\beta_1 {x_2+}\beta_2x_m)x_{a_n-n+1}
\quad\mbox{if $\gamma\in\Gamma_{12}$}\\
&&x_1^{n-2}(\alpha_1x_2+\alpha_2 x_m)(\beta_1 {x_2+}\beta_2x_m)x_{a_1}x_{a_2-1}\cdots x_{a_{n}-n+1}
\quad\mbox{if $\gamma\in\Gamma_{2}$}.
\end{eqnarray*}

Since 
\begin{eqnarray*}
&&1\leq a_1\leq a_2-1\leq\cdots\leq a_n-n+1\leq u-n+1\leq m-1\\
&&a_{n-2}-n+3\leq 2\quad\mbox{if $\gamma\in\Gamma_1$}\\
&&a_{n-1}-n+2\geq 3\quad\mbox{if $\gamma\in\Gamma_2$}\\
\mbox{and}\\
&&a_n-n+1\geq 3\quad\mbox{if $\gamma\in\Gamma_{02}\sqcup\Gamma_{12}$},
\end{eqnarray*}
we see by Lemma \ref{lem:li} that
\begin{eqnarray*}
&&\lf(\delta_X\gamma_{{X}})\quad\mbox{($\gamma\in\Gamma$)}\\
&&\lf((\delta_1)_X\gamma_X)\quad\mbox{($\gamma\in\Gamma\setminus\{\delta\}$)}\\
&&\lf((\delta_2)_X\gamma_X)\quad\mbox{($\gamma\in\Gamma$,
$\supp\gamma\not\supset\{1,\ldots,n-1\}$)}
\end{eqnarray*}
are linearly independent over $K[T]$.
The assertion follows by Lemma \ref{lem:lib}.
\end{proof}

\begin{cor}
\mylabel{cor:cont kt}
Let $\KKK$ be a field, $T$ a $u\times n\times m$ tensor of indeterminates with
$u>n\geq m\geq u-n+2$,
$R$ a commutative ring containing $\KKK(T)$, 
$\xxx=(x_1, \ldots, x_m)$ a vector of indeterminates.
Set $M=M(\xxx,T)$ (c.f.\ Definition \ref{def:unmtv}) and 
$B=R[\xxx]$.
Then $\delta_M$ is a non-zerodivisor of $B$ and for any $k_1$, $k_2$ with
$1\leq k_1, k_2\leq n$, $k_1\neq k_2$, 
\begin{eqnarray*}
&&\delta_M\gamma_M\quad\mbox{($\gamma\in\Gamma$)}\\
&&[1,\ldots,\overset{n+1} {k_2}, \ldots, n]_M\gamma_M\quad\mbox{\rm ($\gamma\in\Gamma\setminus\{\delta\}$)}\\
&&[1,\ldots,\overset{n+1} {k_1}, \ldots, n]_M\gamma_M\quad\mbox{\rm ($\gamma\in\Gamma$,
$\supp\gamma\not\supset\{1,\ldots,\hat k_2, \ldots,n\}$)}
\end{eqnarray*}
are linearly independent over $R$.
\end{cor}
\begin{proof}
Set $A=\KKK[T][\xxx]$.
Then $B=A\otimes _{\KKK[T]}R$.

Since $R$ is flat over $\KKK[T]$, we see that $B$ is flat over $A$.
By Lemma \ref{lem:lct}, we see that
$\delta_M$ is a non-zerodivisor of $A$ and for any $k_1$, $k_2$ with
$1\leq k_1, k_2\leq n$, $k_1\neq k_2$, 
\begin{eqnarray*}
&&\delta_M\gamma_M\quad\mbox{($\gamma\in\Gamma$)}\\
&&[1,\ldots,\overset{n+1} {k_2}, \ldots, n]_M\gamma_M\quad\mbox{\rm ($\gamma\in\Gamma\setminus\{\delta\}$)}\\
&&[1,\ldots,\overset{n+1} {k_1}, \ldots, n]_M\gamma_M\quad\mbox{\rm ($\gamma\in\Gamma$,
$\supp\gamma\not\supset\{1,\ldots,\hat k_2, \ldots,n\}$)}
\end{eqnarray*}
are linearly independent over $\KKK[T]$.
Since $R$ (resp.\ $B$) is flat over $\KKK[T]$ (resp. $A$), we see that
$\delta_M$ is a non-zerodivisor of $B$ and for any $k_1$, $k_2$ with
$1\leq k_1, k_2\leq n$, $k_1\neq k_2$, 
\begin{eqnarray*}
&&\delta_M\gamma_M\quad\mbox{($\gamma\in\Gamma$)}\\
&&[1,\ldots,\overset{n+1} {k_2}, \ldots, n]_M\gamma_M\quad\mbox{\rm ($\gamma\in\Gamma\setminus\{\delta\}$)}\\
&&[1,\ldots,\overset{n+1} {k_1}, \ldots, n]_M\gamma_M\quad\mbox{\rm ($\gamma\in\Gamma$,
$\supp\gamma\not\supset\{1,\ldots,\hat k_2, \ldots,n\}$)}
\end{eqnarray*}
are linearly independent over $R$.
\end{proof}

\begin{cor}
\mylabel{cor:lin indep r}
Let $\xxx=(x_1, \ldots, x_m)$ be a vector of indeterminates.
Suppose that $u>n$ and $Y\in\Aind$ (c.f.\ Definition \ref{def:mu aind psi}) 
and set $M=M(\xxx,Y)$.
Then $\delta_M$ is a non-zerodivisor of $\RRR[\xxx]$ and for any $k_1$, $k_2$ with
$1\leq k_1, k_2\leq n$, $k_1\neq k_2$, 
\begin{eqnarray*}
&&\delta_M\gamma_M\quad\mbox{($\gamma\in\Gamma$)}\\
&&[1,\ldots,\overset{n+1} {k_2}, \ldots, n]_M\gamma_M\quad\mbox{\rm ($\gamma\in\Gamma\setminus\{\delta\}$)}\\
&&[1,\ldots,\overset{n+1} {k_1}, \ldots, n]_M\gamma_M\quad\mbox{\rm ($\gamma\in\Gamma$,
$\supp\gamma\not\supset\{1,\ldots,\hat k_2, \ldots,n\}$)}
\end{eqnarray*}
are linearly independent over $\RRR$.
\end{cor}

\begin{definition}
\rm
\mylabel{def:msqqq'}
Let $\xxx=(x_1, \ldots, x_m)$ be a vector of indeterminates.
We set
$\msQQQ'\define\{Y\in\RRR^{u\times n\times m}\mid
\delta_M\neq 0$ and
$\delta_M\gamma_M\quad\mbox{($\gamma\in\Gamma$)}$,
$[1,\ldots,\overset{n+1} {k_2}, \ldots, n]_M\gamma_M\quad\mbox{\rm ($\gamma\in\Gamma\setminus\{\delta\}$)}$,
$[1,\ldots,\overset{n+1} {k_1}, \ldots, n]_M\gamma_M\quad\mbox{\rm ($\gamma\in\Gamma$,
$\supp\gamma\not\supset\{1,\ldots,\hat k_2, \ldots,n\}$)}$
are linearly independent over $\RRR$ for any $k_1$, $k_2$ with $1\leq k_1<k_2\leq n$,
where $M=M(\xxx,Y)\}$.
\end{definition}

\begin{remark}
\rm
$\msQQQ'$ is a Zariski open set of $\RRR^{u\times n\times m}$ and
by Corollary \ref{cor:lin indep r}, we see that
$\msQQQ'\supset\Aind$.
In particular, $\msQQQ'$ is a Zariski dense open subset of 
$\RRR^{u\times n\times m}$.
\end{remark}

By Lemma \ref{lem:lc}, we see the following fact.

\begin{prop}
\mylabel{prop:lin comb}
Let $u$, $n$ and $m$ be integers with $u>n\geq m\geq u-n+2$
and let $\xxx=(x_1, \ldots, x_m)$ be a vector of indeterminates.
Suppose that $Y\in \msQQQ'$ and 
$\yyy\in\RRR^{1\times n\times m}$.
Set
$$
\tilde Y=\begin{pmatrix}Y\\\yyy\end{pmatrix}.
$$
If $I_n(M(\xxx,\tilde Y))=I_n(M(\xxx,Y))$, then
$\fl_1(\yyy)$ is an $\RRR$-linear combination of rows of $\fl_1(Y)$.
\end{prop}
\begin{proof}
Set $M=M(\xxx,Y)$.
Since $\RRR[\xxx]$ is a domain and $\delta_M\neq 0$ by the definition
of $\msQQQ'$, we see that $\delta_M$ is a non-zerodivisor of $\RRR[\xxx]$.
Moreover,
\begin{eqnarray*}
&&\delta_M\gamma_M\quad\mbox{($\gamma\in\Gamma$)}\\
&&[1,\ldots,\overset{n+1} {k_2}, \ldots, n]_M\gamma_M\quad\mbox{\rm ($\gamma\in\Gamma\setminus\{\delta\}$)}\\
&&[1,\ldots,\overset{n+1} {k_1}, \ldots, n]_M\gamma_M\quad\mbox{\rm ($\gamma\in\Gamma$,
$\supp\gamma\not\supset\{1,\ldots,\hat k_2, \ldots,n\}$)}
\end{eqnarray*}
are linearly independent over $\RRR$ for any $k_1$, $k_2$ with $1\leq k_1<k_2\leq n$
by the definition of $\msQQQ'$.
Thus by Lemma \ref{lem:lc}, we see that 
$M(\xxx,\yyy)$ is an $\RRR$-linear combination of rows of $M=M(\xxx,Y)$.
Since $x_1$, \ldots, $x_m$ are indeterminates,
we see that $\fl_1(\yyy)$ is an $\RRR$-linear combination of rows
of $\fl_1(Y)$.
\end{proof}


\section{Tensor with rank $p$ \mylabel{sec:rank p}}

Let $3\leq m\leq n$, $(m-1)(n-1)+1\leq p\leq (m-1)n$ and set $l=(m-1)n-p$ and $u=n+l$. 
In the following of this paper, we use the results of the previous sections 
by setting $t=n$.
See Definition \ref{def:unmtv}.
Then $v=l+1$ and it follows that $v<m$ since $l\leq m-2$.
Note also $u+p=nm$.
We make bunch of definitions used in the sequel of this paper.

\begin{definition}\rm
\mylabel{def:msVVV sigma etc}
We put 
$$
\msVVV=\msVVV^{n\times p\times m}\define\{T\in\RRR^{n\times p\times m}\mid
\fl_2(T)^{\leq p} \text{ is nonsingular}
\}
$$
and define $\sigma\colon\msVVV\to\RRR^{u\times p}$ be a map defined as
$$\sigma(T)=({}^{p<}\fl_2(T))(\fl_2(T)^{\leq p})^{-1}.$$

We denote by $\msAAA^{u\times n\times m}$\index{$\msAAA^{u\times n\times m}$}  the set of all $u\times n\times m$ \afcr{} tensors and put $\msCCC^{u\times n\times m}=\RRR^{u\times n\times m}\setminus\msAAA^{u\times n\times m}$\index{$\msCCC$}.
Note that $\msAAA^{u\times n\times m}=\msAAA_n^{u\times n\times m}$ and $\msCCC^{u\times n\times m}=\msCCC_n^{u\times n\times m}$ in the notation of
 Definition~\ref{def:mu aind psi}.

\mylabel{def:mgc}
Let $\MMMMM$\index{$\MMMMM$} be the subset of $\RRR^{u\times nm}$ consisting of
all matrices $W=(W_1,\ldots, W_m)$ satisfying that there are
$A=(\aaa_1, \ldots, \aaa_p)\in\RRR^{n\times p}$ and $p\times p$ diagonal matrices $D_1, \ldots, D_m$ such that
$D_k=\diag(d_{1k}, \ldots, d_{pk})$ {for $1\leq k\leq m$, 
$(d_{j1}W_1+\cdots+d_{jm}W_m)\aaa_j=\zerovec$ for $1\leq j\leq p$}, and
\begin{equation}
\mylabel{eqn:def of m}
\begin{pmatrix}AD_1\\ \vdots \\ AD_{m-2}\\ A^{\leq n-l}D_{m-1}\end{pmatrix} 
\end{equation}
is nonsingular.
Let $\iota\colon \RRR^{u\times p}\to\RRR^{u\times nm}$\index{$\iota$} be a map which sends 
$A$ to $(A,-E_{u})$.
Moreover put
$$\begin{array}{l}
\CCCCC\define\{W\in\RRR^{u\times nm}\mid
W\notin\fl_1(\msAAA^{u\times n\times m})\}=
\{W\in\RRR^{u\times n\times m}\mid W\in\fl_1(\msCCC^{u\times n\times m})\}\index{$\CCCCC=\{W\in\RRR^{u\times nm}\mid
W\notin\fl_1(\msAAA^{u\times n\times m})\}$}. \\
\end{array}
$$

\mylabel{def:phi}
We define
$\phi\colon\RRR^{1\times m}\times \RRR^n\to \RRR^{p}$\index{$\phi\colon\RRR^{1\times m}\times \RRR^n\to \RRR^{p}$} a map defined as
$$\phi(\aaa,\bbb)=\begin{pmatrix}a_1\bbb\\ a_2\bbb\\ \vdots\\
a_{m-1}\bbb^{\leq n-l}\end{pmatrix}\in \RRR^{p}, \text{ where } 
\aaa=(a_1,\ldots,a_{m-1},a_m).$$
\end{definition}

\bigbreak

\noindent
Recall that the set $\msAAA^{u\times n\times m}$ is an open subset 
of $\RRR^{u\times n\times m}$ by Lemma \ref{lem:afr open} or Corollary~\ref{cor:msa open}.

\begin{prop} \mylabel{prop:sigma open}
$\sigma$ is an open, surjective {and continuous} map.
\end{prop}

\begin{proof}
Clearly $\sigma$ is surjective {and continuous}.
Let $\msOOO$ be an open subset of $\msVVV$ 
and let $h\colon \RRR^{nm\times p} \to \RRR^{p \times p}\times\RRR^{u \times p}$ be a homeomorphism defined as $h(M)=(M^{\leq p}, {}^{p<}M)$.
Then $h(\fl_2(\msOOO))$ can be written by 
$$\bigcup_\lambda O_{1,\lambda}\times O_{2,\lambda}$$
for some open subsets $O_{1,\lambda}\subset \RRR^{p\times p}$ and
$O_{2,\lambda}\subset \RRR^{u\times p}$
and thus 
$$\sigma(\msOOO)=\bigcup_\lambda \bigcup_{A\in O_{1,\lambda}} O_{2,\lambda}A^{-1}.$$
The set $O_{2,\lambda}A^{-1}$ is open and then $\sigma(\msOOO)$ is open.
\end{proof}

The following fact follows from the definition.

\begin{lemma}
$\MMMMM$ is stable under the action of $\glin(u,\RRR)$.
\end{lemma}

\begin{lemma}
\mylabel{lem:m sub c}
$\MMMMM\subset\CCCCC$.
\end{lemma}

\begin{proof}
By observing the first column of \eqref{eqn:def of m}, we see that $\aaa_1\neq\zerovec$ and
at least one of $d_{11}$, $d_{12}$, \ldots, $d_{1,m-1}$ is nonzero,
since 
$\begin{pmatrix}AD_1\\ \vdots \\ AD_{m-2}\\ A^{\leq n-l}D_{m-1}\end{pmatrix}$
is nonsingular, where $D_k=\diag(d_{1k},\ldots,d_{pk})$ for $1\leq k\leq m-1$.
Since $d_{11}W_1+\cdots+d_{1m}W_m$ is singular, $(W_1;\ldots;W_m)$ is not \afcr, 
i.e., $\MMMMM\subset\CCCCC$.
\end{proof}

\begin{thm} \mylabel{thm:equiv cond}
Let 
$X\in \msVVV\subset\RRR^{n\times p\times m}$.
Put $(W_1,\ldots,W_{m-1},W_m)=\iota(\sigma(X))$, where $W_1,\ldots,W_{m}\in\RRR^{u\times n}$.
The following four statements are equivalent.
\begin{enumerate}
\item \mylabel{eq:rank} $\rank X=p$.
\item \mylabel{eq:perfect} There are an $n\times p$ matrix $A$, and diagonal $p\times p$ matrices $D_1,\ldots,D_m$ such that
\begin{equation}\mylabel{eq:matrix}
W_1AD_1+W_2AD_2+\cdots+W_{m-1}AD_{m-1}+W_mAD_{m}=O
\end{equation}
and 
\begin{equation}\mylabel{eq:Qinv}
N=\begin{pmatrix} AD_1 \\ \vdots \\ AD_{m-2}\\ A^{\leq n-l}D_{m-1}\end{pmatrix}
\end{equation}
is nonsingular.
\item \mylabel{eq:inM} $\iota(\sigma(X))\in\MMMMM$.
\end{enumerate}
\end{thm}

\begin{proof}
It holds that $\rank X\geq p$, since $\fl_2(X)^{\leq p}$ has rank $p$.
Put $(S_1;\ldots;S_m)=X(\fl_2(X)^{\leq p})^{-1}$
Then $\rank X=\rank (S_1;\ldots;S_m)$ and
$$
\begin{pmatrix} {}^{n-l<}S_{m-1}\\ S_m\end{pmatrix}=
(W_1,W_2,\ldots,W_{m-2},(W_{m-1})_{\leq n-l}).
$$

\ref{eq:rank} $\Rightarrow$ \ref{eq:perfect}:
Since $\rank (S_1;\ldots;S_m)=p$, there are an $n\times p$-matrix $A$,
a $p\times p$-matrix $Q$, and $p\times p$ diagonal matrices
$D_1, \ldots, D_m$ such that
$$
AD_kQ=S_k\quad\mbox{for $k=1, \ldots, m$.}
$$

Since
$$
NQ=
\begin{pmatrix}
AD_1\\ \vdots \\ AD_{m-2} \\ A^{\leq n-l}D_{m-1}
\end{pmatrix}
Q=
\begin{pmatrix}
S_1\\\vdots\\ S_{m-2}\\ (S_{m-1})^{\leq n-l}
\end{pmatrix}
=E_p,
$$
we see that $N$ and $Q$ are nonsingular and $Q^{-1}=N$.
Since 
$$
\begin{pmatrix} {}^{n-l<}S_{m-1}\\ S_m\end{pmatrix}=
\sigma(X)=
(W_1,W_2,\ldots,W_{m-2},(W_{m-1})_{\leq n-l}),
$$
and $AD_k=S_kN$ for $k=m-1$, $m$, we see that

\begin{eqnarray*}
&&O_{u\times p}\\
&=&
\begin{pmatrix} {}^{n-l<}S_{m-1}\\ S_m\end{pmatrix}N
-\begin{pmatrix} {}^{n-l<}AD_{m-1}\\ AD_m\end{pmatrix} \\
&=& W_1AD_1+\cdots+W_{m-2}AD_{m-2} \\
&& \qquad +(W_{m-1})_{\leq n-l}A^{\leq n-l}D_{m-1}
+\begin{pmatrix}-E_l\\O\end{pmatrix}{}^{n-l<}AD_{m-1}
+\begin{pmatrix}O\\-E_n\end{pmatrix}AD_m\\
&=&W_1AD_1+\cdots +W_{m-2}AD_{m-2}+W_{m-1}AD_{m-1}+W_mAD_m.
\end{eqnarray*}
Therefore the equation \eqref{eq:matrix} holds.

\ref{eq:perfect} $\Rightarrow$ \ref{eq:rank}:
Set $Q=N^{-1}$.
Then, since $NQ=E_p$, we see that 
$$
AD_kQ=S_k\quad\mbox{$1\leq k\leq m-2$}
$$
and
$$
A^{\leq n-l}D_{m-1}Q=S_{m-1}^{\leq n-l}.
$$
Furthermore, since 
\begin{eqnarray*}
&& W_1AD_1+\cdots+W_{m-2}AD_{m-2} \\
&&\qquad +(W_{m-1})_{\leq n-l}A^{\leq n-l}D_{m-1}
+\begin{pmatrix}-E_l\\O\end{pmatrix}{}^{n-l<}AD_{m-1}
+\begin{pmatrix}O\\-E_n\end{pmatrix}AD_m\\
&=&W_1AD_1+\cdots+W_mAD_m\\
&=&O_{u\times p},
\end{eqnarray*}
we see that 
$$
\begin{pmatrix} {}^{n-l<}S_{m-1}\\ S_m\end{pmatrix}N
{=\begin{pmatrix}E_l\\ O\end{pmatrix}{}^{n-l<}AD_{m-1}+\begin{pmatrix}O\\ E_n\end{pmatrix}AD_m}.
$$
Thus
$$
\begin{pmatrix}{}^{n-l<}AD_{m-1}\\AD_{m}\end{pmatrix}Q=
\begin{pmatrix} {}^{n-l<}S_{m-1}\\ S_m\end{pmatrix}
$$
and we see that
$AD_kQ=S_k$ for $k=m-1$, $m$.
Therefore, $\rank X=\rank (S_1;\ldots;S_m)\leq p$
and we see \ref{eq:rank}.

Finally it is easy to see that \ref{eq:perfect} $\Leftrightarrow$ \ref{eq:inM}.
\end{proof}

Since $\MMMMM\subset\CCCCC$ by Lemma~\ref{lem:m sub c}, we see the following:

\begin{prop}\mylabel{prop:afr}
For $X\in \msVVV$, if $\rank X=p$, then $\iota(\sigma(X))\not\in \fl_1(\msAAA^{u\times n\times m})$.
\end{prop}


\section{Contribution of \afcr\ property}
\mylabel{sec:cor}

Let $m$, $n$ and $p$ be integers with 
$3\leq m\leq n$ and $(m-1)(n-1)+1\leq p\leq (m-1)n$.
We set $u=nm-p$
and $t=n$ and 
use the results of Sections \ref{sec:im} and \ref{sec:mpo}.
Note $v=u-n+1=(m-1)n-p+1$ in the notation of Definition \ref{def:unmtv}.

{It is known that the generic rank $\grank(n,p,m)$ of $n\times p\times m$ tensors over 
$\CCC$ is equal to $p$ (\cite[Theorem 3.1]{Catalisano-Geramita-Gimigliano:2002} 
or \cite[Theorem 2.4 and Remark 2.5]{Catalisano-Geramita-Gimigliano:2008}) 
and it is also equal to the minimal typical rank of $n\times p\times m$ tensors over $\RRR$.
Thus if we discuss the plurality of typical ranks, 
it is enough to consider whether there exists a typical rank that is 
greater than $p$ or not.}

\begin{definition}
\rm
\mylabel{def:a'}
We set $\AAAAA\define \iota^{-1}(\fl_1(\msAAA^{u\times n\times m}))\subset\RRR^{u\times p}$\index{$\AAAAA\define \iota^{-1}(\fl_1(\msAAA^{u\times n\times m}))$}, 
where $\iota$ and $\msAAA^{u\times n\times m}$ are defined in Definition \ref{def:msVVV sigma etc}.
\end{definition}

\begin{lemma}
\mylabel{lem:a' rank}
If $Y\in\msVVV^{n\times p\times m}$ and $\sigma(Y)\in \AAAAA$, then 
$\rank Y>p$.
\end{lemma}
\begin{proof}
This follows from the fact that 
$\rank Y\geq p$ if $Y\in\msVVV^{n\times p\times m}$ and
Proposition \ref{prop:afr}.
\end{proof}


\begin{thm}
\mylabel{thm:afr exist}
If $\msAAA^{u\times n\times m}\neq \emptyset$, then
there are plural typical ranks of $n\times p\times m$ tensors 
over $\RRR$.
\end{thm}
\begin{proof}
By Lemma \ref{lem:equivalent condition for afr},
we see that $\AAAAA\neq\emptyset$.
Since $\msAAA^{u\times n\times m}$ is an open subset of $\RRR^{u\times n\times m}$,
we see that $\AAAAA$ is an open subset of $\RRR^{u\times p}$.
{Moreover}, since $\sigma\colon\msVVV^{{n\times p\times m}}\to\RRR^{u\times p}$ is a surjective continuous map, we see that ${\sigma^{-1}(\AAAAA)}$ is a nonempty open subset of $\msVVV^{{n\times p\times m}}$.
Thus, there is a typical rank greater than $p$ by Lemma \ref{lem:a' rank}.

Since $p$ is a typical rank of $n\times p\times m$ tensors over $\RRR$,
we see that there are plural typical ranks of $n\times p\times m$ tensors
over $\RRR$.
\end{proof}

From now on until the end of this section, we assume
that $p\geq (m-1)(n-1)+2$.
Thus, $m\geq v+2$.

\begin{definition}\rm
Let $Y\in\RRR^{u\times n\times m}$
and let $\xxx=(x_1, \ldots, x_m)$ be a vector of indeterminates.
For $i_1$, \ldots, $i_{n-1}\in\{1,\ldots, u\}$, we set
$$
\psi_{i_1, \ldots, i_{n-1}}(\xxx,Y)\define
\begin{pmatrix}
(-1)^{n+1}[i_1, \ldots, i_{n-1}\mid 2, \ldots, n-1, n]_{M(\xxx,Y)}\\
(-1)^{n+2}[i_1, \ldots, i_{n-1}\mid 1,3, \ldots, n-1, n]_{M(\xxx,Y)}\\
\vdots\\
(-1)^{2n}[i_1, \ldots, i_{n-1}\mid 1, \ldots, n-2, n-1]_{M(\xxx,Y)}
\end{pmatrix}\in \RRR[\xxx]^n.
$$
For the definition
$[a_1, \ldots, a_t\mid b_1,\ldots,b_t]$, see Definition \ref{def:unmtv}.
We define $\hat\psi_{i_1,\ldots, i_{n-1}}\colon \RRR^{1\times m} \times 
\RRR^{u\times n\times m}\to \RRR[\xxx]^p$ 
by
$$
\hat\psi_{i_1,\ldots, i_{n-1}}(\xxx,Y)\define
\begin{pmatrix}x_1\psi_{i_1,\ldots, i_{n-1}}(\xxx,Y)\\ 
x_2\psi_{i_1,\ldots, i_{n-1}}(\xxx,Y)\\ \vdots\\
x_{m-1}\psi_{i_1,\ldots, i_{n-1}}(\xxx,Y)\end{pmatrix}^{\leq p}
\in\RRR[\xxx]^p.
$$
We also define the 
$\RRR$-vector space $U(Y)$ by
$$
U(Y)\define\langle\hat\psi_{i_1,\ldots, i_{n-1}}(\uuu,Y)\mid 
\begin{array}{l}
\uuu\in\VVV(I_n(M(\xxx,Y))), \\
i_1, \ldots, i_{n-1}\in\{1,\ldots, u\}
\end{array}
\rangle
\subset\RRR^p.$$

For 
$\ccc=(c_{11},\ldots,c_{n1},c_{12},\ldots,c_{n2},\ldots,c_{1m},\ldots,c_{nm})
\in\RRR^{1\times nm}$,
we set
$Z_k=\begin{pmatrix}Y_k\\ c_{1k}\ \cdots\ c_{nk}\end{pmatrix}$ for $1\leq k\leq m$,
$Z=(Z_1;\ldots;Z_m)\in\RRR^{(u+1)\times n\times m}$ and
$$
\begin{array}{l}
g_{i_1, \ldots, i_{n-1}}(\xxx,Y,\ccc)=[i_1, \ldots, i_{n-1},u+1]_{M(\xxx,Z)}
\end{array}$$
for any $i_1$, \ldots, $i_{n-1}\in\{1,\ldots, u\}$.
For the definition
$[i_1, \ldots, i_{n-1},u+1]_{M(\xxx,Z)}$, see
Definition \ref{def:unmtv}.
\end{definition}


\begin{lemma}
\mylabel{lem:span equiv}
Suppose that $Y\in \RRR^{u\times n\times m}$. Then
the following claims are equivalent.
\begin{enumerate}
\item
\mylabel{item:uy}
$\dim U(Y)=p$.
\item
\mylabel{item:guyc0}
If $\ccc\in\RRR^{1\times nm}$ satisfies 
the following conditions, then $\ccc=\zerovec$.
\begin{itemize}
\item[$(\ast)$]
${}_{p<}\ccc=\zerovec$ and
\item[$(\ast\ast)$]
$g_{i_1, \ldots i_{n-1}}(\uuu,Y,\ccc)=0$ for any $\uuu\in\VVV(I_n(M(\xxx,Y)))$
and any $i_1$, \ldots, $i_{n-1}\in\{1,\ldots, u\}$.
\end{itemize}
\end{enumerate}
\end{lemma}

\begin{proof}
The vector $\ddd\in\RRR^{p}$ is perpendicular to $U(Y)$ if and only if $\ddd$ is perpendicular to
$\hat\psi_{i_1, \ldots, i_{n-1}}(\uuu,Y)$ for any $\uuu\in\VVV(I_n(M(\xxx,Y)))$
and any $i_1, \ldots, i_{n-1}$.
Since the inner product of $\hat\psi_{i_1, \ldots, i_n}(\uuu,Y)$ with $\ddd$
is 
$ g_{i_1, \ldots, i_{n-1}}(\uuu,Y,(\ddd\transpose,\zerovec))$, 
the result follows.
\end{proof}

Next we show the following result.
For the definition of $\MMMMM$, see Definition~\ref{def:mgc}.

\begin{lemma}
\mylabel{lem:span suf}
If $\dim U(Y)=p$, then $\fl_1(Y)\in \MMMMM$.
\end{lemma}

\begin{proof}
Set $Y=(Y_1;\ldots;Y_m)$.
Suppose that $\dim U(Y)=p$.
Then there are $\uuu_1, \ldots, \uuu_p\in\VVV(I_n(M(\xxx,Y)))$
and $t_{11}$, \ldots, $t_{1,n-1}$, \ldots, $t_{p1}$, \ldots, $t_{p,n-1}$ 
such that
$$\hat\psi_{t_{11},\ldots,t_{1,n-1}}(\uuu_1,Y), 
\ldots, \hat\psi_{t_{p1}, \ldots, t_{p,n-1}}(\uuu_p,Y)$$
are linearly independent over $\RRR$.
Set
$\uuu_j=(u_{j1}, \ldots, u_{jm})$ for $1\leq j\leq p$,
$D_k=\diag(u_{1k}, \ldots, u_{pk})$ for $1\leq k\leq m$ and
$$A=(\psi_{t_{11}, \ldots, t_{1,n-1}}(\uuu_1,Y), \ldots, 
\psi_{t_{p1}, \ldots, t_{p,n-1}}(\uuu_p,Y)).$$
Then,
$$
\begin{pmatrix}
AD_1\\ \vdots\\ AD_{m-2}\\ A^{\leq n-l}D_{m-1}
\end{pmatrix}
=\begin{pmatrix}
AD_1\\ \vdots\\ AD_{m-2}\\ AD_{m-1}
\end{pmatrix}^{\leq p}
=(\hat\psi_{t_{11}, \ldots, t_{1,n-1}}(\uuu_1,Y),\ldots, \hat\psi_{t_{p1}, \ldots, t_{p,n-1}}(\uuu_p,Y))
$$
is a nonsingular matrix and
$$
(u_{j1}Y_1+\cdots+u_{jm}Y_m)\psi_{t_{j1},\ldots,t_{j,n-1}}(\uuu_{j},Y)
=\begin{pmatrix}
[1,t_{j1}, \ldots, t_{j,n-1}]_{M(\uuu_j,Y)}\\
[2,t_{j1}, \ldots, t_{j,n-1}]_{M(\uuu_j,Y)}\\
\vdots\\
[u,t_{j1}, \ldots, t_{j,n-1}]_{M(\uuu_j,Y)}
\end{pmatrix}
=\zerovec
$$
since $I_n(M(\uuu_j,Y))=(0)$ for $1\leq j\leq p$.
\end{proof}


\begin{definition}
\rm
Set
$\msUUU\define\{Y\in\RRR^{u\times n\times m}\mid
{}_{p<}\fl_1(Y)$ is nonsingular$\}$,
$\msOOO_3\define\msUUU\cap\msQQQ\cap\msQQQ'\cap
\tilde\msPPP_n
=\msOOO_1\cap\msUUU\cap\msQQQ'$\index{$\msOOO_3=\msUUU\cap\msQQQ\cap\msQQQ'\cap
\tilde\msPPP_n$}
and
$\msOOO_4\define\msUUU\cap\msQQQ\cap\msQQQ'\cap\msAAA^{u\times n\times m}
=\msOOO_2\cap\msUUU\cap\msQQQ'$\index{$\msOOO_4=\msUUU\cap\msQQQ\cap\msQQQ'\cap\msAAA^{u\times n\times m}$},
where $\msQQQ$, $\msQQQ'$ and 
$\tilde\msPPP_n$
are the ones defined in
Definitions \ref{def:msqqq}, \ref{def:msqqq'}, and \ref{def:tilde p}
and $\msOOO_1$ and $\msOOO_2$ are the ones in Theorem \ref{thm:it real} {under $t=n$}.
Define ${\nu}\colon\msUUU\to\RRR^{u\times p}$\index{$\nu\colon\msUUU\to\RRR^{u\times p}$} as
${\nu}(Y)\define -({}_{p<}\fl_1(Y))^{-1}\fl_1(Y)_{\leq p}$
for $i=1, 2$,
where $\sigma$ is the one defined in Definition \ref{def:msVVV sigma etc}.
Set $\OOOOO_i={\nu}(\msOOO_{i+2})\subset\RRR^{u\times p}$%
\index{$\OOOOO_i={\nu}(\msOOO_{i+2})$} 
{and} $\msTTT_i=\sigma^{-1}(\OOOOO_i){\subset\msVVV^{n\times p\times m}}$%
\index{$\msTTT_i=\sigma^{-1}(\OOOOO_i)$}
for $i=1, 2$. 
\end{definition}

The following fact is immediately verified.

\begin{lemma}
$\iota(\nu(Y))=\fl_1(-({}_{p<}\fl_1(Y))^{-1}Y)$.
\end{lemma}

By the same way as Theorem \ref{thm:it real}
\ref{item:msooo open},
\ref{item:msooo dense} and
\ref{item:closure o1},
we see the following fact.

\begin{lemma}
\mylabel{lem:o' property}
Then the followings hold.
\begin{enumerate}
\item
\mylabel{item:msooo' open}
$\msOOO_3$ and $\msOOO_4$ are disjoint open subsets of $\RRR^{u\times n\times m}$ and $\msOOO_3$ is nonempty.
\item
\mylabel{item:msooo' dense}
$\msOOO_3\cup\msOOO_4$ is a dense subset of $\RRR^{u\times n\times m}$.
\item
\mylabel{item:closure o'1}
$\overline{\msOOO_3}=\msCCC^{u\times n\times m}=\RRR^{u\times n\times m}\setminus\msAAA^{u\times n\times m}$.
\end{enumerate}
\end{lemma}

\begin{lemma}
\mylabel{lem:c 0}
$\fl_1(\msOOO_3)\subset\MMMMM$.
\end{lemma}
\begin{proof}
Let $Y\in\msOOO_3$.
By Lemmas \ref{lem:span suf} and \ref{lem:span equiv}, it is enough to show that
if $\ccc\in\RRR^{1\times nm}$ satisfies $(\ast)$ and $(\ast\ast)$, then
$\ccc=\zerovec$.

Set $\ccc=(\ccc_1, \ldots, \ccc_m)$, where $\ccc_j\in\RRR^{1\times n}$,
$\ccc'=(\ccc_1;\cdots;\ccc_m)\in\RRR^{1\times n\times m}$
and $\tilde Y=\begin{pmatrix}Y\\ \ccc'\end{pmatrix}$.
Then by $(\ast\ast)$,
$g_{i_1, \ldots, i_{n-1}}(\xxx,Y,\ccc)\in\III(\VVV(I_n(M(\xxx,Y))))$
for any $i_1$, \ldots, $i_{n-1}\in\{1$, \ldots, $u\}$.
Therefore, by the definition of $\msOOO_3$ and Theorem \ref{thm:it real}
\ref{item:real closed}, we see that 
$g_{i_1, \ldots, i_{n-1}}(\xxx,Y,\ccc)\in I_n(M(\xxx,Y))$
for any $i_1$, \ldots, $i_{n-1}\in\{1$, \ldots, $u\}$.
Thus we see that $I_n(M(\xxx,\tilde Y))=I_n(M(\xxx,Y))$.
Thus, by Proposition \ref{prop:lin comb}, we see that $\fl_1(\ccc)$ is 
an $\RRR$-linear combination of rows of $\fl_1(Y)$.
Since ${}_{p<}\ccc=\zerovec$ and $Y\in\msUUU$, we see that $\ccc=\zerovec$.
\end{proof}

By the same way as Proposition \ref{prop:sigma open}, we see the following:

\begin{prop}
\mylabel{prop:nu open}
${\nu}$ is an open, {surjective and continuous} map.
\end{prop}

We see the following fact.

\begin{lemma}
\mylabel{lem:ooooo property}
Then the followings hold.
\begin{enumerate}
\item
\mylabel{item:msaaa aaaaa}
$Y\in\msAAA^{u\times n\times m}$ if and only if
$\nu(Y)\in\AAAAA$ for $Y\in\msUUU$.
\item
\mylabel{item:ooooo open}
$\OOOOO_1$ and $\OOOOO_2$ are disjoint open subsets of $\RRR^{u\times p}$ and $\OOOOO_1\neq\emptyset$.
\item
\mylabel{item:ooooo dense}
$\OOOOO_1\cup\OOOOO_2$ is a dense subset of $\RRR^{u\times p}$.
\item
\mylabel{item:closure ooooo1}
$\overline{\OOOOO_1}=\RRR^{u\times p}\setminus\AAAAA$. 
\item
\mylabel{item:ooooo2 subset a'}
$\OOOOO_2\subset\AAAAA$ and $\overline{\OOOOO_2}=\overline{\AAAAA}$.
\end{enumerate}
\end{lemma}
\begin{proof}
\ref{item:msaaa aaaaa}:
Suppose that $Y\in\msUUU$.
Since $\msAAA^{u\times n\times m}$ and $\msUUU$ are stable under the action 
of $\glin(u,\RRR)$, we see that $Y\in\msAAA^{u\times n\times m}$ if and only if 
$-({}_{p<}\fl_1(Y))^{-1}Y\in\msAAA^{u\times n\times m}$.
Since $\iota(\nu(Y))=\fl_1(-({}_{p<}\fl_1(Y))^{-1}Y)$
and $\fl_1$ is a bijection, we see
\ref{item:msaaa aaaaa}.

We see 
\ref{item:ooooo open} 
by the facts that $\tilde \msPPP_n$ and $\msAAA^{u\times n\times m}$ are
stable under the action of  $\glin(u,\RRR)$,
Lemma \ref{lem:o' property}, and Proposition \ref{prop:nu open}.
\ref{item:ooooo dense} also follows from Lemma \ref{lem:o' property}
 and Proposition \ref{prop:nu open}.
We see by \ref{item:msaaa aaaaa}
that if $Y\in\msOOO_3$, then $\nu(Y)\not\in\AAAAA$.
Thus $\OOOOO_1=\nu(\msOOO_3)\subset\RRR^{u\times p}\setminus\AAAAA$.
Since $\overline{\msOOO_3}=\RRR^{u\times n\times m}\setminus \msAAA^{u\times n\times m}$
by Lemmas \ref{lem:o' property} \ref{item:closure o'1},
we see that $\overline{\OOOOO_1}\supset\nu(\overline{\msOOO_3}\cap\msUUU)
=\nu((\RRR^{u\times n\times m}\setminus\msAAA^{u\times n\times m})\cap\msUUU)
=\RRR^{u\times p}\setminus\AAAAA$
by \ref{item:msaaa aaaaa} and the surjectivity of $\nu$.
Thus we see \ref{item:closure ooooo1}.
Therefore $\OOOOO_2\subset\AAAAA$ by \ref{item:ooooo open}.
Further, we see that $\overline\OOOOO_2\supset\AAAAA$ by \ref{item:ooooo dense}
and \ref{item:closure ooooo1}.
Thus we see \ref{item:ooooo2 subset a'}.
\end{proof}
%
%

\begin{lemma}
\mylabel{lem:inv closure}
Let $X$ and $Y$ be topological spaces, 
$f\colon X\to Y$ a mapping and $B$ a subset of $Y$.
\begin{enumerate}
\item
\mylabel{item:cont}
If $f$ is continuous, then
$f^{-1}(\overline B)\supset \overline{f^{-1}(B)}$.
\item
\mylabel{item:open}
If $f$ is an open map, then
$f^{-1}(\overline B)\subset \overline{f^{-1}(B)}$.
\end{enumerate}
\end{lemma}
\begin{proof}
\ref{item:cont}:
Since $f^{-1}(\overline B)$ is a closed subset of $X$ containing $f^{-1}(B)$,
we see that $f^{-1}(\overline B)\supset \overline{f^{-1}(B)}$.

\ref{item:open}:
Suppose that $x\in f^{-1}(\overline B)$ and let $U$ be an open
neighborhood of $x$.
We show that $U\cap f^{-1}(B)\neq\emptyset$.

Since $f(x)\in \overline B$, $f(x)\in f(U)$ and $f(U)$ is an open subset of $Y$,
we see that $f(U)\cap B\neq\emptyset$.
Take $b\in f(U)\cap B$ and $a\in U$ such that $f(a)=b$.
Then, since $f(a)\in B$, we see that $a\in f^{-1}(B)$.
Thus, $a\in U\cap f^{-1}(B)$ and we see that $U\cap f^{-1}(B)\neq\emptyset$.
\end{proof}

\begin{thm}
\mylabel{thm:rank}
{Let $m$, $n$ and $p$ be integers with 
$3\leq m\leq n$ and $(m-1)(n-1)+2\leq p\leq (m-1)n$.}
The followings hold.
\begin{enumerate}
\item
\mylabel{item:msttt open}
$\msTTT_1$ and $\msTTT_2$ are disjoint open subsets of $\msVVV^{n\times p\times m}$ and $\msTTT_1$ is nonempty.
\item
\mylabel{item:msttt dense}
$\msTTT_1\cup\msTTT_2$ is a dense subset of $\RRR^{n\times p\times m}$.
\item
\mylabel{item:closure msttt1}
$\overline{\msTTT_1}\cap\msVVV^{n\times p\times m}=\msVVV^{n\times p\times m}\setminus{\sigma^{-1}(\AAAAA)}$ and $\overline{\msTTT_2}\cap\msVVV^{n\times p\times m}=\overline{{\sigma^{-1}(\AAAAA)}}\cap\msVVV^{n\times p\times m}$.
\item
\mylabel{item:msttt1 rank}
If $T\in\msTTT_1$, then $\rank T=p$.
\item
\mylabel{item:msttt2 rank}
If $T\in\msTTT_2$, then $\rank T>p$.
\end{enumerate}
\end{thm}
\begin{proof}
First note that 
$\overline{\sigma^{-1}(\XXXXX)}\cap\msVVV^{n\times p\times m}=\sigma^{-1}(\overline \XXXXX)$ 
for any subset $\XXXXX$ of $\RRR^{u\times p}$
by Lemma \ref{lem:inv closure},
since $\sigma$ is an open continuous map.

\ref{item:msttt open} and \ref{item:msttt dense} follow from Lemma \ref{lem:ooooo property}
and the facts that $\sigma$ is surjective and $\msVVV^{n\times p\times m}$ is a dense subset
of $\RRR^{n\times p\times m}$.

\bgroup

\ref{item:closure msttt1}:
We see by Lemma \ref{lem:ooooo property} that
$\overline\msTTT_1\cap\msVVV^{n\times p\times m}=\sigma^{-1}(\overline{\OOOOO_1})=
\sigma^{-1}(\RRR^{u\times p}\setminus\AAAAA)=\msVVV^{n\times p\times m}
\setminus\sigma^{-1}(\AAAAA)$
and
$\overline\msTTT_2\cap\msVVV^{n\times p\times m}
=\sigma^{-1}(\overline{\OOOOO_2})=\overline{\sigma^{-1}(\AAAAA)}
\cap\msVVV^{n\times p\times m}$.

\ref{item:msttt1 rank}:
Suppose that $T\in\msTTT_1$.
Then $\sigma(T)\in\OOOOO_1$. Thus there exists $Y\in\msOOO_3$ such that ${\nu}(Y)=\sigma(T)$.
By Lemma \ref{lem:c 0}, we see that $\fl_1(Y)\in\MMMMM$.
{Hence} $\iota(\sigma(T))=\iota({\nu}(Y))=-({}_{p<}\fl_1(Y))^{-1}\fl_1(Y)\in\MMMMM$,
since $\MMMMM$ is stable under the action of $\glin(u,\RRR)$.
Therefore $\rank T=p$ by Theorem \ref{thm:equiv cond}.

\ref{item:msttt2 rank}:
If $T\in\msTTT_2$, then $\sigma(T)\in\OOOOO_2\subset\AAAAA$ by Lemma \ref{lem:ooooo property}.
Thus $\rank T>p$ by Lemma \ref{lem:a' rank}.
\egroup
\end{proof}

\section{Upper bound for typical ranks\mylabel{sec:upperbound}}

In Lemma \ref{lem:a' rank},
we see a class of tensors with rank greater than $p$.
To complete the proof of Theorem~\ref{thm:main1}, we give an upper bound of the set of typical ranks of $\RRR^{n\times p\times m}${:}

\begin{thm} \mylabel{thm:upperbound}
Let $3\leq m\leq n$ {and $(m-1)(n-1)+1\leq p\leq (m-1)n$}.
Any typical rank of $\RRR^{n\times p\times m}$ is less than or equal to
$p+1$.
\end{thm}

We prepare the proof.

Let $3\leq m\leq n$, $(m-1)(n-1)+1\leq p<(m-1)n$ and $u=mn-p$.
Let 
$\sigma^\prime\colon\msVVV^{n\times (p+1)\times m}\to\RRR^{(u-1)\times (p+1)}$ be the 
counterpart of 
$\sigma\colon\msVVV^{n\times p\times m}\to\RRR^{u\times p}$. 
Also, let $\AAAAA'\subset\RRR^{(u-1)\times (p+1)}$ and 
$\msTTT'_1\subset\msVVV^{n\times (p+1)\times m}$ be the counterparts of
$\AAAAA\subset\RRR^{u\times p}$ and $\msTTT_1\subset\msVVV^{n\times p\times m}$ respectively.
Let $\pi\colon\RRR^{n\times (p+1)\times m}\to\RRR^{n\times p\times m}$ be a canonical projection defined as
$\pi(Y_1;\ldots;Y_m)=((Y_1)_{\leq p};\ldots;(Y_m)_{\leq p})$.
Clearly $\pi$ is a continuous, surjective and open map.

\begin{lemma}
\mylabel{lem:pi ttt1' open dense}
$\pi(\msTTT'_1)$ is an open dense subset of $\RRR^{n\times p\times m}$.
\end{lemma}

\begin{proof}
Since $\msTTT'_1$ is an open set and $\pi$ is an open map, $\pi(\msTTT'_1)$ is an open subset of $\RRR^{n\times p\times m}$.
We show that $\pi(\msTTT'_1)$ is dense.
Let $X\in\msVVV^{n\times p\times m}$.
Consider the map $f\colon \msVVV^{n\times p\times m}\to \msVVV^{n\times (p+1)\times m}$ defined as 
$$f(X_1;\ldots;X_{m-2};X_{m-1};X_m)
=((X_1,\zerovec);\ldots;(X_{m-2},\zerovec);(X_{m-1},\eee);(X_m,\zerovec))$$
where $\eee$ is the $(2n-u+1)$th column vector of the identity matrix $E_n$.
Since the $(p+1)$th column vector of the matrix $\sigma'(f(X))$ is zero, $f(X)\notin \sigma'^{-1}(\AAAAA')$ holds and by Theorem~\ref{thm:rank} \ref{item:closure msttt1}, $f(X)\in \overline{\msTTT'_1}$.
Since $\pi\circ f$ is the identity map and $\pi$ is continuous, $X\in \pi(\overline{\msTTT'_1})\subset\overline{\pi(\msTTT'_1)}$ holds.
Therefore $\msVVV^{n\times p\times m}\subset \overline{\pi(\msTTT'_1)}$ and thus $\RRR^{n\times p\times m}=\overline{\pi(\msTTT'_1)}$.
\end{proof}

By Theorem~\ref{thm:rank} \ref{item:msttt2 rank},
and Lemma~\ref{lem:pi ttt1' open dense}, 
we have immediately the following corollary.

\begin{cor} \mylabel{cor:class p+1}
Let $3\leq m\leq n$ and $(m-1)(n-1)+1\leq p<(m-1)n$.
$\msTTT_2\neq\emptyset$ if and only if $\msTTT_2\cap \pi(\msTTT'_1)\neq\emptyset$, and $\rank T=p+1$ for any $T\in\msTTT_2\cap \pi(\msTTT'_1)$.
\end{cor}

Note that arbitrary tensor of $\pi(\msTTT'_1)$ has rank less than or equal to $p+1$ by Theorem~\ref{thm:rank} \ref{item:msttt1 rank}. 

\medskip
\begin{proofof}{Theorem~\ref{thm:upperbound}}
The assertion for $p=(m-1)n$ holds by \cite{Sumi-etal:2015}.
Suppose that $(m-1)(n-1)+1\leq p<(m-1)n$.
Then $\rank(T)\leq p+1$ for $T\in \pi(\msTTT'_1)$.  Since $\pi(\msTTT'_1)$ is dense, arbitrary integer greater than $p+1$ is not a typical rank.
\end{proofof}

Recall that $\trank(m,n,p)=\trank(n,p,m)$.
We are ready to prove main theorems.

\medskip
\begin{proofof}{Theorem~\ref{thm:main0}}
\ref{item:01} follows from Theorem \ref{thm:afr exist} and 
Corollary \ref{cor:afr}.

\ref{item:02}:
We may assume that $3\leq m\leq n$ without the loss of generality.
Ten Berge \cite{tenBerge:2000} showed that $\RRR^{m\times n\times p}$
has a unique typical rank for $p\geq (m-1)n+1$.
Therefore, we see that $p\leq (m-1)n$.
Set $u=mn-p$.
By Theorem~\ref{thm:rank} \ref{item:msttt dense} and \ref{item:msttt1 rank},
we see that $\msTTT_2\neq\emptyset$.
Furthermore, $\msTTT_2\neq\emptyset \Rightarrow \msOOO_4\neq\emptyset$ by definitions 
and the surjectivity of 
$\sigma$.
Since $\msOOO_4\subset\msAAA^{u\times n\times m}$, we see that there exists an
\afcr\ $u\times n\times m$ tensor.
The result follows from Corollary \ref{cor:afr}.
\end{proofof}

\begin{proofof}{Theorem~\ref{thm:main1}}
We may assume that $3\leq m\leq n$.
Note that 
$$\typicalrankR(m,n,p)=\{{\min\{p,mn\}}\}$$
for {$k\geq m$} \cite{tenBerge:2000}.
Suppose that $2\leq k\leq m-1$.  
By Theorem~\ref{thm:upperbound}, the maximal typical rank of $\RRR^{m\times n\times p}$ is less than or equal to $p+1$.  Since $p$ is the minimal typical rank of $\RRR^{m\times n\times p}$, $\typicalrankR(m,n,p)$ is $\{p\}$ or $\{p,p+1\}$.
By Theorem~\ref{thm:main0}, 
$\RRR^{m\times n\times p}$ has a unique typical rank if and only if $m\# n{\geq mn-p+1}$, equivalently, $k{\geq m+n-(m\# n)}$.
This completes the proof.
\end{proofof}

We immediately have Theorem~\ref{thm:(m-1)(n-1)+1} by 
Proposition~\ref{prop:bit-disjoint} and {Theorem~\ref{thm:afr exist}}.
In the case where $p=(m-1)(n-1)+1$, we have many examples for having plural typical ranks.

\begin{cor}
Let $m,n\geq 3$ and $a\geq 1$.
If {$m\equiv 2^{a-1}+s \pmod{2^a}$ and $n\equiv 2^{a-1}+t \pmod{2^a}$ for some integers $s$ and $t$ with $1\leq s,t\leq 2^{a-1}$} then $\RRR^{m\times n\times ((m-1)(n-1)+1)}$ has plural typical ranks.
\end{cor}

\begin{prop} 
Let ${a}=4,8$.  If $m$ and $n$ are divisible by ${a}$, then for each $1\leq {k<a}$,
$\RRR^{m\times n\times ((m-1)(n-1)+{k})}$ has plural typical ranks.
\end{prop}

\begin{proof}
For ${a}=4,8$, if $m$ and $n$ are divisible by ${a}$, then $m\#n\leq m+n-{a}$ 
by Proposition~\ref{prop:4m-4n} and thus $m+n-1-(m\#n)\geq {a}-1$.
Then the assertion follows by Theorem~\ref{thm:main1}.
\end{proof}

\begin{cor}
\begin{enumerate}
\item $\RRR^{4\times 4\times k}$ has plural typical rank whenever
$10\leq k\leq 12$.
\item $\RRR^{8\times 8\times k}$ has plural typical rank whenever
$50\leq k\leq 56$.
\end{enumerate}
\end{cor}

\bgroup
\begin{prop}\label{prop:(m-1)(n-1)}
Let $m,n\geq 3$.
If $\RRR^{m\times n\times ((m-1)(n-1)+1)}$ has a unique typical rank, i.e., 
$m\#n=m+n-1$, then
$\typicalrankR(m,n,(m-1)(n-1)-k)=\{(m-1)(n-1)+1\}$ holds whenever $0\leq k<\frac{(m-1)(n-1)}{m+n-1}$.
\end{prop}

\begin{proof}
Let $0\leq k<\frac{(m-1)(n-1)}{m+n-1}$, $q=(m-1)(n-1)-k$ and $p=(m-1)(n-1)+1$.
Suppose that $\RRR^{m\times n\times p}$ has a unique typical rank.
Then $\trank(n,p,m)=\{p\}$.
Since the set of all $n\times p\times m$ tensors with rank $p$ is a dense subset of $\RRR^{n\times p\times m}$, the image of this set by a canonical projection $\RRR^{n\times p\times m}\to\RRR^{n\times q\times m}$ is also a dense subset of $\RRR^{n\times q\times m}$.  Thus any typical rank of $\RRR^{n\times q\times m}$ is less than or equal to $p$.
On the other hand, by elementary calculation, we see that $(m-1)(n-1)<\frac{mnq}{m+n+q-2} \Leftrightarrow k<\frac{(m-1)(n-1)}{m+n-1}$.
Thus the minimal typical rank of $\RRR^{n\times q\times m}$ is greater than or equal to $p$.
Therefore $\RRR^{n\times q\times m}$ has a unique typical rank $p$.
\end{proof}

\begin{cor} \label{cor:fromThmC}
Let $3\leq m\leq n$. 
Suppose that $\RRR^{m\times n\times ((m-1)(n-1)+1)}$ has a unique typical rank,
i.e., $m\#n=m+n-1$.
If $0\leq k\leq \lfloor \frac{m}{2}\rfloor-1$ then 
$\typicalrankR(m,n,(m-1)(n-1)-k)=\{(m-1)(n-1)+1\}$.
\end{cor}

\begin{proof}
Let $0\leq k\leq \lfloor \frac{m}{2}\rfloor-1$.
Then $(m+n-1)(k+1)\leq (m+n-1)\frac{m}{2}\leq(n+n-1)\frac{m}{2}<mn$ and thus
$(m+n-1)k<(m-1)(n-1)$.
Therefore the assertion follows from Proposition~\ref{prop:(m-1)(n-1)}.
\end{proof}

\egroup


\iffalse
\nocite{Catalisano-Geramita-Gimigliano:2002}
\bibliography{../bibs/tensor}

\begin{thebibliography}{10}

\bibitem{Adams:1962}
J.~F. Adams.
\newblock Vector fields on spheres.
\newblock {\em Ann. of Math. (2)}, 75:603--632, 1962.

\bibitem{Adem:1968}
J.~Adem.
\newblock Some immersions associated with bilinear maps.
\newblock {\em Bol. Soc. Mat. Mexicana (2)}, 13:95--104, 1968.

\bibitem{Atiyah-Macdonald:1969}
M.~F. Atiyah and I.~G. Macdonald.
\newblock {\em Introduction to commutative algebra}.
\newblock Addison-Wesley Publishing Co., Reading, Mass.-London-Don Mills, Ont.,
  1969.

\bibitem{Bruns-Vetter:1988}
W.~Bruns and U.~Vetter.
\newblock {\em Determinantal rings}, volume 1327 of {\em Lecture Notes in
  Mathematics}.
\newblock Springer-Verlag, Berlin, 1988.

\bibitem{Catalisano-Geramita-Gimigliano:2002}
M.~V. Catalisano, A.~V. Geramita, and A.~Gimigliano.
\newblock Ranks of tensors, secant varieties of {S}egre varieties and fat
  points.
\newblock {\em Linear Algebra Appl.}, 355:263--285, 2002; 
erratum: 367:347--348, 2003.

\bibitem{Catalisano-Geramita-Gimigliano:2008}
M.~V. Catalisano, A.~V. Geramita, and A.~Gimigliano.
\newblock On the ideals of secant varieties to certain rational varieties.
\newblock {\em J. Algebra}, 319(5):1913--1931, 2008.

\bibitem{Cohen:1985}
R.~L. Cohen.
\newblock The immersion conjecture for differentiable manifolds.
\newblock {\em Ann. of Math. (2)}, 122(2):237--328, 1985.

\bibitem{Davis:1984}
D.~M. Davis.
\newblock A strong non-immersion theorem for real projective spaces.
\newblock {\em Ann. of Math.}, 120(3):517--528, 1984.

\bibitem{Davis:2011}
D.~M. Davis.
\newblock Some new nonimmersion results for real projective spaces.
\newblock {\em Bol. Soc. Mat. Mexicana (3)}, 17(2):159--166, 2011.

\bibitem{Davis-Mahowald:2008}
D.~M. Davis and M.~Mahowald.
\newblock Nonimmersions of {${\mathbf R}{\rm P}^n$} implied by tmf, revisited.
\newblock {\em Homology, Homotopy Appl.}, 10(3):151--179, 2008.

\bibitem{Eisenbud-etal:1992}
D.~Eisenbud, C.~Huneke, and W.~Vasconcelos.
\newblock Direct methods for primary decomposition.
\newblock {\em Invent. Math.}, 110(2):207--235, 1992.

\bibitem{Friedland:2012}
S.~Friedland.
\newblock On the generic and typical ranks of 3-tensors.
\newblock {\em Linear Algebra Appl.}, 436(3):478--497, 2012.

\bibitem{Ginsburg:1963}
M.~Ginsburg.
\newblock Some immersions of projective space in {E}uclidean space.
\newblock {\em Topology}, 2:69--71, 1963.

\bibitem{Heintz-Sieveking:1981}
J.~Heintz and M.~Sieveking.
\newblock Absolute primality of polynomials is decidable in random polynomial
  time in the number of variables.
\newblock In {\em Automata, languages and programming ({A}kko, 1981)}, volume
  115 of {\em Lecture Notes in Comput. Sci.}, pages 16--28. Springer,
  Berlin-New York, 1981.

\bibitem{Hitchcock:1927}
F.~L. Hitchcock.
\newblock The expression of a tensor or a polyadic as a sum of products.
\newblock {\em J. Math. Phys.}, 6(1):164--189, 1927.

\bibitem{Hochster-Eagon:1971}
M.~Hochster and J.~A. Eagon.
\newblock Cohen-{M}acaulay rings, invariant theory, and the generic perfection
  of determinantal loci.
\newblock {\em Amer. J. Math.}, 93:1020--1058, 1971.

\bibitem{Hurwitz:1922}
A.~Hurwitz.
\newblock \"{U}ber die {K}omposition der quadratischen {F}ormen.
\newblock {\em Math. Ann.}, 88(1-2):1--25, 1922.

\bibitem{JaJa:1979b}
J.~Ja'Ja'.
\newblock Optimal evaluation of pairs of bilinear forms.
\newblock {\em SIAM J. Comput.}, 8(3):443--462, 1979.

\bibitem{Kemper:2014}
G.~Kemper and N.~Viet~Trung.
\newblock Krull dimension and monomial orders.
\newblock {\em J. Algebra}, 399:782--800, 2014.

\bibitem{Lam:1968a}
K.~Y. Lam.
\newblock Construction of some nonsingular bilinear maps.
\newblock {\em Bol. Soc. Mat. Mexicana (2)}, 13:88--94, 1968.

\bibitem{Levine:1963}
J.~Levine.
\newblock Imbedding and immersion of real projective spaces.
\newblock {\em Proc. Amer. Math. Soc.}, 14:801--803, 1963.

\bibitem{Matsumura:1989}
H.~Matsumura.
\newblock {\em Commutative ring theory}, volume~8 of {\em Cambridge Studies in
  Advanced Mathematics}.
\newblock Cambridge University Press, Cambridge, second edition, 1989.
\newblock Translated from the Japanese by M. Reid.

\bibitem{Milgram:1967}
R.~J. Milgram.
\newblock Immersing projective spaces.
\newblock {\em Annals of Mathematics}, 85(3):473--482, 1967.

\bibitem{Miyazaki-etal:2012a}
M.~Miyazaki, T.~Sumi, and T.~Sakata.
\newblock Typical ranks of certain 3-tensors and absolutely full column rank
  tensors.
\newblock {\em preprint, arXiv:1103.0154v2}, Dec. 2012.

\bibitem{Nagata:1993}
M.~Nagata.
\newblock {\em Theory of commutative fields}, volume 125 of {\em Translations
  of Mathematical Monographs}.
\newblock American Mathematical Society, Providence, RI, 1993.
\newblock Translated from the 1985 Japanese edition by the author.

\bibitem{Radon:1922}
J.~Radon.
\newblock Lineare scharen orthogonaler matrizen.
\newblock {\em Abh. Math. Sem. Univ. Hamburg}, 1(1):1--14, 1922.

\bibitem{Rees:1957}
D.~Rees.
\newblock The grade of an ideal or module.
\newblock {\em Proc. Cambridge Philos. Soc.}, 53:28--42, 1957.

\bibitem{Robbiano:1985}
L.~Robbiano.
\newblock Term orderings on the polynomial ring.
\newblock In {\em E{UROCAL} '85, {V}ol.\ 2 ({L}inz, 1985)}, volume 204 of {\em
  Lecture Notes in Comput. Sci.}, pages 513--517. Springer, Berlin, 1985.

\bibitem{Seidenberg:1974}
A.~Seidenberg.
\newblock Constructions in algebra.
\newblock {\em Trans. Amer. Math. Soc.}, 197:273--313, 1974.

\bibitem{Shapiro:2000}
D.~B. Shapiro.
\newblock {\em Compositions of quadratic forms}, volume~33 of {\em de Gruyter
  Expositions in Mathematics}.
\newblock Walter de Gruyter \& Co., Berlin, 2000.

\bibitem{Singh:2004}
N.~Singh.
\newblock On nonimmersion of real projective spaces.
\newblock {\em Topology and its Applications}, 136:233--238, 2004.

\bibitem{Sumi-etal:2009}
T.~Sumi, M.~Miyazaki, and T.~Sakata.
\newblock Rank of $3$-tensors with $2$ slices and {K}ronecker canonical forms.
\newblock {\em Linear Algebra Appl.}, 431(10):1858--1868, 2009.

\bibitem{Sumi-etal:2015}
T.~Sumi, M.~Miyazaki, and T.~Sakata.
\newblock Typical ranks of {$m\times n\times (m-1)n$} tensors with {$3\leq
  m\leq n$} over the real number field.
\newblock {\em Linear Multilinear Algebra}, 63(5):940--955, 2015.

\bibitem{Sumi-etal:2013}
T.~Sumi, T.~Sakata, and M.~Miyazaki.
\newblock Typical ranks for {$m\times n\times(m-1)n$} tensors with {$m\leq n$}.
\newblock {\em Linear Algebra Appl.}, 438(2):953--958, 2013.

\bibitem{tenBerge:2000}
J.~M.~F. ten Berge.
\newblock The typical rank of tall three-way arrays.
\newblock {\em Psychometrika}, 65(4):525--532, December 2000.

\bibitem{tenBerge-Kiers:1999}
J.~M.~F. ten Berge and H.~A.~L. Kiers.
\newblock Simplicity of core arrays in three-way principal component analysis
  and the typical rank of {$p\times q\times 2$} arrays.
\newblock {\em Linear Algebra Appl.}, 294(1-3):169--179, 1999.

\end{thebibliography}
\bibliographystyle{abbrv}
\else

\fi



\end{document}